\documentclass[12pt]{amsart}
\title{Curve-excluding fields}
\author{Will Johnson and Jinhe Ye}
\email{willjohnson@fudan.edu.cn, jinhe.ye@maths.ox.ac.uk}

\usepackage{amsmath, amssymb, amsthm, enumitem, comment}
\usepackage[Symbol]{upgreek}
\usepackage{tikz-cd}
\usepackage[mathscr]{euscript}
\usepackage{fullpage} 	% without this, will have wide math-article-paper margins
\usepackage{hyperref}
\usepackage{lineno}
\usepackage[normalem]{ulem}
\usepackage[all]{xy}
\usepackage{centernot}
\usepackage{xcolor}

\DeclareFontFamily{U}{BOONDOX-calo}{\skewchar\font=45 }
\DeclareFontShape{U}{BOONDOX-calo}{m}{n}{
  <-> s*[1.05] BOONDOX-r-calo}{}
\DeclareFontShape{U}{BOONDOX-calo}{b}{n}{
  <-> s*[1.05] BOONDOX-b-calo}{}
\DeclareMathAlphabet{\mathcalboondox}{U}{BOONDOX-calo}{m}{n}
\SetMathAlphabet{\mathcalboondox}{bold}{U}{BOONDOX-calo}{b}{n}
\DeclareMathAlphabet{\mathbcalboondox}{U}{BOONDOX-calo}{b}{n}

\newcommand{\pdeg}{\mathrm{pdeg}}

\newcommand{\trdg}{\operatorname{tr.\!dg}}

\newcommand{\Gal}{\operatorname{Gal}}
\newcommand{\Spec}{\operatorname{Spec}}

\newcommand{\Aut}{\operatorname{Aut}}
\newcommand{\Frac}{\operatorname{Frac}}
\newcommand{\acl}{\operatorname{acl}}
\newcommand{\Th}{\operatorname{Th}}

\newcommand{\dcl}{\operatorname{dcl}}
\newcommand{\tp}{\operatorname{tp}}
\newcommand{\qftp}{\operatorname{qftp}}

\newcommand{\XF}{\mathrm{XF}}
\newcommand{\Abs}{\mathrm{Abs}}
\newcommand{\alg}{\mathrm{alg}}
\newcommand{\Res}{\mathrm{Res}}
\newcommand{\clring}{\mathcal{L}_{\mathrm{ring}}}
\newcommand{\ACF}{\mathrm{ACF}}
\newcommand{\PAC}{\mathrm{PAC}}
\DeclareMathOperator{\id}{id}

\newtheorem*{claim-star}{Claim}
\newtheorem{theorem}{Theorem}[section] % numbered like the section
\newtheorem{lemma}[theorem]{Lemma}

\newtheorem{prop-def}[theorem]{Proposition-Definition}
\newtheorem{corollary}[theorem]{Corollary}
\newtheorem{fact}[theorem]{Fact}
\newtheorem{fact-eh}[theorem]{Fact(?)}

\newtheorem{conjecture}[theorem]{Conjecture}

\newtheorem{proposition}[theorem]{Proposition}
\newtheorem{proposition-eh}[theorem]{Proposition(?)}
\newtheorem*{theorem-star}{Theorem}
\newtheorem*{conjecture-star}{Conjecture}
\newtheorem*{question-star}{Question}
\newtheorem*{lemma-star}{Lemma}

\theoremstyle{definition}
\newtheorem{definition}[theorem]{Definition}
\newtheorem{example}[theorem]{Example}

\newtheorem{question}[theorem]{Question}

\newtheorem{remark}[theorem]{Remark}
\theoremstyle{remark}
\newtheorem{claim}[theorem]{Claim}

\newtheorem*{acknowledgment}{Acknowledgments}

\newtheorem*{warning}{Warning}

\newcommand{\Aa}{\mathbb{A}}

\newcommand{\Qq}{\mathbb{Q}}
\newcommand{\Rr}{\mathbb{R}}
\newcommand{\Zz}{\mathbb{Z}}

\newcommand{\Pp}{\mathbb{P}}

\newcommand{\Mm}{\mathbb{M}}

\newcommand{\cL}{\mathcal{L}}

\newcommand{\NSOP}{\mathrm{NSOP}}

\newenvironment{claimproof}[1][\proofname]
               {
                 \proof[#1]
                 
               }
               {
                 \endproof
               }

\def\Ind#1#2{#1\setbox0=\hbox{$#1x$}\kern\wd0\hbox to 0pt{\hss$#1\mid$\hss}
\lower.9\ht0\hbox to 0pt{\hss$#1\smile$\hss}\kern\wd0}

\def\ind{\mathop{\mathpalette\Ind{}}}

\def\notind#1#2{#1\setbox0=\hbox{$#1x$}\kern\wd0
\hbox to 0pt{\mathchardef\nn=12854\hss$#1\nn$\kern1.4\wd0\hss}
\hbox to 0pt{\hss$#1\mid$\hss}\lower.9\ht0 \hbox to 0pt{\hss$#1\smile$\hss}\kern\wd0}

\def\nind{\mathop{\mathpalette\notind{}}}

\begin{document}

\begin{abstract}
  If $C$ is a curve over $\Qq$ with genus at least $2$ and $C(\Qq)$ is empty, then the class
  of fields $K$ of characteristic 0 such that $C(K) = \varnothing$ has
  a model companion, which we call $C\XF$.  The theory $C\XF$ is not complete, but we characterize the completions.  Using $C\XF$, we produce
  examples of fields with interesting combinations of properties.  For
  example, we produce (1) a model-complete field with unbounded Galois
  group, (2) an infinite field with a decidable first-order theory
  that is not ``large'' in the sense of Pop, (3) a field that is
  algebraically bounded but not ``very slim'' in the sense of Junker
  and Koenigsmann, and (4) a pure field that is strictly NSOP$_4$, i.e.,
  NSOP$_4$ but not NSOP$_3$.  Lastly, we give a new construction
  of fields that are virtually large but not large.
\end{abstract}

\maketitle

\section{Introduction}
\subsection{Motivation}
The model theory of fields has been long-studied. It is a classical
theorem of Tarski~\cite[Theorem 2.7.3]{Hodges} that algebraically
closed fields have quantifier elimination in the language of rings ($\clring$), and
Macintyre showed that no other infinite fields have quantifier
elimination in $\clring$~\cite{Macintyre-omegastable}.

In contrast, the class of model complete fields is much richer.
Recall that a theory $T$ is \emph{model complete} if every formula is
equivalent modulo $T$ to an existential formula. This is the immediate logical weakening of quantifier elimination.  A structure $M$ is
\emph{model complete} if its complete theory $\Th(M)$ is model
complete.  Many important theories of fields are model complete,
including algebraically closed fields, real closed fields,
$p$-adically closed fields, and some pseudofinite fields. Based on
these examples, Macintyre asked the following question:
\begin{question} \label{q-1}
  If a field $K$ is model complete (in the language of rings), must
  the Galois group $\Gal(K)$ be bounded?
\end{question}
Here, we say that $\Gal(K)$ is \emph{bounded} if $K$ has finitely many
separable extensions of degree $n$ for each $n$, or equivalently,
$\Gal(K)$ has finitely many open subgroups of index $n$ for each $n$.
When $K$ is highly saturated, this condition is equivalent to
$|\Gal(K)| \le 2^{\aleph_0}$.

A related question was asked by Junker and Koenigsmann~\cite{JK-slim}:
\begin{question} \label{q-2}
  If $\Gal(K)$ is bounded and $K$ is infinite, must $K$ be large?
\end{question}
Here, a field $K$ is \emph{large} in the sense of Pop~\cite{pop-embedding} if every smooth curve $C$ over $K$ with at least one $K$-rational point has infinitely many $K$-rational points. Large fields are also called \emph{ample} or \emph{anti-Mordellic}.  To the best of our knowledge, all of
the fields considered ``tame" by model theorists are large, including the
examples mentioned above and much beyond.  In contrast, number fields like $\Qq$ fail
to be large by Faltings' theorem.

Chaining together Questions~\ref{q-1} and \ref{q-2}, we get the
following, which also appeared in~\cite{JK-slim} and \cite{secondpaper}.
\begin{question} \label{q-3}
  If a field $K$ is model complete (in the language of rings), must
  $K$ be large?
\end{question}
Question~\ref{q-3} dovetails well with recent work on stable large fields \cite{firstpaper} and simple large fields \cite{with-anand}.
A positive answer to this question would have several interesting consequences.  For example, it would imply that all model complete fields are
\'ez
  in the sense of~\cite{secondpaper}, which would enable a uniform
  topological approach to model complete fields.  A negative answer would show that the connection between ``logical tameness'' (e.g. model completeness) and ``field-theoretic tameness'' (e.g. largeness) is more subtle than expected.

As it happens, there is a standard method to construct model complete
theories.  Recall that a structure $M$ is \emph{existentially closed
(e.c.)} in an extension $N$, written $M \prec_\exists N$, if
\begin{equation*}
  N \models \phi \implies M \models \phi
\end{equation*}
for existential sentences $\phi$ with parameters in $M$.  For a fixed theory $T$, a model $M
\models T$ is \emph{e.c.} if it is e.c.\@ in any extension $N
\supseteq M$ satisfying $T$.

There is a close connection between model completeness and e.c.\@
models.  For example, $T$ is model complete if and only if all models
are e.c.  Recall that a theory $T$ is \emph{inductive} if any union of
a chain of models is a model, or equivalently, $T$ is axiomatized by
$\forall\exists$-sentences.
\begin{fact} \label{mc-fact}
  Let $T$ be a consistent inductive theory.
  \begin{enumerate}
  \item Every model of $T$ embeds into an e.c.\@ model.  In
    particular, e.c.\@ models exist.
  \item Suppose there is a theory $T^{\mathrm{ec}}$ axiomatizing the
    class of e.c.\@ models of $T$.  Then $T^{\mathrm{ec}}$ is
    consistent and model complete.
  \end{enumerate}
\end{fact}
In part (2), $T^{\mathrm{ec}}$ is called the \emph{model companion} of
$T$. In this case, we also say \emph{the model companion of $T$ exists} or $T$ \emph{has a model companion.} For a more general treatment of model companions, see~\cite[Chapter 8]{Hodges}.

This suggests how one might construct a counterexample to
Question~\ref{q-3}.  Let $C$ be any smooth curve over $\mathbb{Q}$ with finitely many rational points. Let $T$ be the theory of fields $K$
of characteristic 0 such that there are no $K$-rational points on $C$
other than the given $\Qq$-rational points.  No model of $T$ is large.
If $T$ has a model companion $T^{\mathrm{ec}}$, then any model $M
\models T^{\mathrm{ec}}$ will be model complete, but not large.  Luckily for us, the model companion exists whenever $C$ has genus $\ge 2$.  For example, one can take $C$ to be the affine curve $x^4 + y^4 = 1$, the fourth Fermat curve, whose only rational points are $\{(\pm 1, 0), (0, \pm 1)\}$.

This gives a negative answer to Question \ref{q-3}.  This means that one of
Question~\ref{q-1} and Question~\ref{q-2} must have a negative answer.
By carefully analyzing the counterexample, we get a negative answer to
Macintyre's question: 

\begin{theorem}[{= Corollary~\ref{cor:non-large-model-complete}, Theorem~\ref{thm:unbdd}}]
  There is a model complete field $K$ with an unbounded Galois group,
  such that $K$ is not large.
\end{theorem}
On the other hand, Question~\ref{q-2} remains open.

\subsection{The theory $C$XF in general}
We now describe our construction in more detail and greater
generality, giving an outline of out results.  Let $K_0$ be a field of characteristic 0 and let $C$ be a
projective or affine curve over $K_0$ with genus at least 2 and no $K_0$-rational points.  Let
$T_{C,K_0}$ be the theory of fields $K/K_0$ with $C(K) = \varnothing$.  When
$K_0 = \Qq$, we can work in the language of rings, but in general we
should work in the language of rings with added constant symbols for
the elements of $K_0$.
\begin{theorem}[{= Theorem~\ref{thm:cxf-exist}, Remark~\ref{rem:computable}}]
  $T_{C,K_0}$ has a computably axiomatized model companion $C\XF_{K_0}$.
\end{theorem}
We will omit the subscript $K_0$ when it is clear from context, just writing $C\XF$.
\begin{warning}
\begin{enumerate}
\item In the example mentioned earlier, the field $K_0$ is $\Qq$
and the curve $C$ is not the Fermat curve $C_4$, but rather the open
subvariety $C_4 \setminus \{(\pm 1, 0), (0, \pm 1)\}$.
\item In spite of our original motivation, models of $C$XF may or may not be large.  We will say more about this shortly.
\end{enumerate}
\end{warning}

We now summarize the fundamental properties of $C\XF$, beginning with
the classification of its completions.  If $K$ extends $K_0$, let
$\Abs(K)$ denote the relative algebraic closure of $K_0$ in $K$.
\begin{theorem}\label{thm:abs}
  \begin{enumerate}
  \item If $M_1,M_2 \models C\XF$, then $M_1 \equiv M_2$ if and only
    if $\Abs(M_1) \cong \Abs(M_2)$ over $K_0$. (Theorem~\ref{completions-1}).
  \item If $K$ is an algebraic extension of $K_0$ such that $K \models
    T_{C,K_0}$ (meaning $C(K) = \varnothing$), then there is some $M \models
    C\XF$ with $\Abs(M) \cong K$. (Theorem~\ref{completions-2}).
  \end{enumerate}
\end{theorem}
This can be used to produce decidable models of $C\XF$.  Recall that a field $K$ is \emph{decidable} if there is an algorithm that decides whether a first-order sentence is true in $K$.  If a field is decidable, then Hilbert's 10th problem is solvable over $K$---there is an algorithm to tell whether a given system of equations has a solution in $K$.  There are many examples of decidable fields, including algebraically closed, separably closed, real closed, and $p$-adically closed fields. In fact, most ``logically tame'' theories of fields have (some) decidable models, and the same happens with $C$XF.  For example, if $K_0 = \Qq$ and $\Abs(M) =
\Qq$, then $\Th(M)$ is decidable because it is computably axiomatized by $C\XF$ plus the condition ``$\Abs(M) = \Qq$'' (see the proof of Theorem~\ref{thm:fermat-example}).

One can also ask whether the incomplete theory $C$XF is decidable.  This question---which we do not resolve---is related to the effective Mordell problem in number theory; see Theorem~\ref{thm:mordell}.

Second, we characterize whether the models of $C\XF$ are large.  We let
$\tilde{C}$ denote the projective completion of $C$, that is, the
projective smooth curve that is birational to $C$.
\begin{theorem}[{= Theorem~\ref{large-characterize}}]
  Suppose $M \models C\XF$.  Then $M$ is large if and only if
  $\tilde{C}(M) = \varnothing$.
\end{theorem}
For instance, in the Fermat curve example, $\tilde{C}(M)$ always
contains the four $\Qq$-rational points on the Fermat curve, and so the models are never large.  However, if we replace the Fermat curve with $x^4 + y^4 = -1$, then the resulting models are large.

Combining the decidability and non-largeness, we get the following:
\begin{corollary}[{= Corollary~\ref{cor:decidable-non-large}}]
  There is an infinite, decidable, non-large field $M$.
\end{corollary}
We believe this is the first known construction of a decidable non-large field.
  
Aside from (non-)largeness, models of $C\XF$ have the following field-theoretic properties:
\begin{theorem} \label{field-theoretic}
  Suppose $M \models C\XF$.
  \begin{enumerate}
  \item $M$ is Hilbertian (Theorem~\ref{Hilb}).
  \item $\Gal(M)$ is unbounded  (Theorem~\ref{thm:unbdd}).  In fact, $\Gal(M)$ is $\omega$-free (Theorem~\ref{omega-free}).
  \item If $L/M$ is a proper algebraic extension, then $L$ is a pseudo-algebraically closed (PAC) field (Theorem~\ref{PAC}).
  \item If $L/M$ is a proper finite extension, then $L$ is an
    $\omega$-free PAC field (Theorem~\ref{proper-finite-ext}).
  \end{enumerate}
\end{theorem}
For the definition of PAC fields and $\omega$-free profinite groups and more detailed discussions on them, see \cite{field-arithmetic}.
Recall that a field $K$ is \emph{virtually
large} if some finite extension of $K$ is large.
Theorem~\ref{field-theoretic}(4) shows that models of $C\XF$ are
virtually large, because $\omega$-free PAC fields are large.  This gives a new construction of virtually large
fields that are not large.  A previous construction was found by Srinivasan~\cite{Srinivasan}.

Next, we turn to the model-theoretic properties of $C\XF$:
\begin{theorem} \phantomsection \label{model-theoretic}  
  \begin{enumerate}
  \item $C\XF$ has quantifier elimination after adding the predicates
    \begin{equation*}
      Sol_n(x_0,\ldots,x_n) \iff \exists y ~ (x_ny^n + \cdots + x_1y + x_0 = 0)
    \end{equation*}
    (Corollary~\ref{cor:near-qe}).
  \item $C\XF$ is an algebraically bounded, geometric theory.  More precisely,
    \begin{enumerate}
    \item Model-theoretic algebraic closure in $C\XF$ agrees with
      field-theoretic algebraic closure over $K_0$.
    \item Model-theoretic algebraic closure satisfies exchange.
    \item The $\exists^\infty$ quantifier is uniformly eliminated.
    \end{enumerate}
    See Theorem~\ref{acl} and Corollary~\ref{exists-infty}.
  \item $C\XF$ does not eliminate imaginaries.  In fact, if $M \models
    C\XF$, then $M^\times/(M^\times)^2$ is not in interpretable
    bijection with any definable set (Remark~\ref{no-imaginaries}).
  \end{enumerate}
\end{theorem}
Junker and Koenigsmann~\cite{JK-slim} say that a field is \emph{very
slim} if model-theoretic and field-theoretic algebraic closure agree
in elementarily equivalent fields.  When $K_0 = \Qq$,
Theorem~\ref{model-theoretic} shows that models of $C\XF$ are very
slim.  On the other hand, when $K_0$ contains transcendentals, models
of $C\XF$ can fail to be very slim.  This gives an example of a field
that is algebraically bounded, but not very slim
(Example~\ref{eg:alg-bdd}).

The position of $C\XF$ in the classification-theoretic universe is
mostly but not completely known.  At present, we know the following:
\begin{theorem}[{= Theorem~\ref{nsop4}, Corollary~\ref{tp2}}] \label{nsop4-tp2}
  If $M \models C\XF$, then $M$ is $\mathrm{NSOP}_4$ and $\mathrm{TP}_2$.
\end{theorem}
\begin{theorem}[{= Theorem~\ref{thm:some-sop3}}] \label{sop3}
  There is \emph{some} field $K_0$, \emph{some} curve $C$, and
  \emph{some} model $M \models C\XF_{K_0}$ such that $M$ has $\mathrm{SOP}_3$.
\end{theorem}
We review the definition of $\mathrm{SOP}_n$ in Section~\ref{nsop4-section} and $\mathrm{TP}_2$ in Section~\ref{sec:other-properties}.
We conjecture that Theorem~\ref{sop3} holds for \emph{any} model of
$C\XF$.  This would mean that $C\XF$ has similar
classification-theoretic properties to Henson's triangle-free graph~\cite{hensen-graph},
being $\mathrm{TP}_2$, $\mathrm{NSOP}_4$, and $\mathrm{SOP}_3$.  For a complete discussion of the $\mathrm{SOP}_n$ hierarchy, see~\cite{shelah-sop3,shelah-sop1}.

At any rate, Theorem~\ref{sop3} has the following consequence, which
seems to be new:
\begin{corollary}[{= Corollary~\ref{strict-nsop4}}]
  There is a pure field $(K,+,\cdot)$ that is strictly $\mathrm{NSOP}_4$,
  meaning that $K$ is $\mathrm{NSOP}_4$ but has $\mathrm{SOP}_3$.
\end{corollary}

\subsection*{Conventions}
A \emph{variety} is a reduced separated scheme of finite type over a field
$K$, not necessarily integral.  We will restrict our attention to characteristic 0, so reducedness implies geometric reducedness \cite[Tag~020I]{stacks-project}.  
  If $V$ is a variety over $K$, and $L$ extends
$K$, then $V(L)$ denotes the set of $L$-points on $V$.  A \emph{curve}
over $K$ is a smooth and geometrically integral $K$-variety of
dimension 1, not necessarily projective. 
If $V$ is an integral $K$-variety, then $K(V)$
denotes the function field.

We use $\deg(f)$ or $\deg(X/Y)$ to denote the degree of a dominant rational map
$f:X\dashrightarrow Y$ between integral varieties, which is the degree of the field extension
$[K(X):K(Y)]$ given by $f$. For projective varieties $X\subseteq
\Pp^n$, we write the
projective degree as $\pdeg(X)$.  Then $\pdeg(X)$ is the number of
intersections of $X$ with a generic projective subspace of codimension
equal to $\dim(X)$.

For an introduction to model theory, see~\cite{Hodges}.  A
\emph{theory} is a collection of sentences, not necessarily complete
or deductively closed.  The \emph{language of rings} is
$\clring=\{0,1,+,\cdot,-\}$.  ``Definable'' means ``definable with
parameters,'' and ``0-definable'' means ``definable without
parameters.'' Lastly, we use $\acl(-)$ to denote the model-theoretic algebraic closure and $(-)^\alg$ to denote the field-theoretic algebraic closure.

\section{Existence of the model companion}\label{sec:axiom}

We only consider fields of characteristic zero from this point on.
Let $K_0$ be a field and $C$ be a curve over $K_0$ with genus
$g(C)\geq 2$, such that $C(K_0)$ is empty.  By a \emph{$K_0$-field},
we mean a field extending $K_0$.  We consider $K_0$-fields as
$\cL(K_0)$-structures, where $\cL(K_0)$ is the expansion of $\clring$
with constants for $K_0$.  Say that a $K_0$-field $K$ \emph{avoids}
$C$ if $C(K_0) = \varnothing$.  Let $T_{C,K_0}$ denote the theory of
$K_0$-fields avoiding $C$.  This is clearly an inductive theory,
meaning that unions of chains of models are models.
\begin{theorem} \label{thm:cxf-exist}
  The theory $T_{C,K_0}$ has a model companion, which we call $C\XF_{K_0}$ or $C\XF$.
\end{theorem}
We divide the proof of the theorem into several steps.  First, we
characterize the e.c.\@ models of $T_{C,K_0}$.  We will need the following three facts:
\begin{fact}
\label{fact:ec}

Suppose that $V$ is an integral $K$-variety.
Then $K \prec_\exists K(V)$ if and only if $V(K)$ is Zariski dense in $V$.
\end{fact}
Here and in what follows, ``$V(K)$ is Zariski dense in $V$'' means that the natural injection from $K$-rational points to scheme-theoretic points has dense image.  Or more concretely, it means that $U(K)$ is non-empty for any non-empty open subvariety $U \subseteq V$.  Fact~\ref{fact:ec} is well-known; see~\cite[Fact 2.3]{Pop-little} for
example.
\begin{lemma}
  \label{lem:fg}
  Let $T$ be a theory of fields such that any subfield of a model is a
  model.  If $K \models T$, then the following are equivalent:
  \begin{enumerate}
  \item $K$ is an e.c.\@ model of $T$.
  \item $K$ is e.c.\@ in any finitely generated extension $L$ of $K$ with $L \models T$.
\end{enumerate}
\end{lemma}

\begin{proof}
(1) trivially implies (2).
Suppose that $L \models T$ is an extension of $K$.
Fix a quantifier-free formula $\phi(x,y)$ and $a \in K^m, b \in L^n$ such that $L \models \phi(a,b)$.
Let $L' = K(a,b)$, which is a finitely generated model of $T$.  Note that $L' \models \phi(a,b)$.
By (2) there is $b' \in K^n$ such that $K \models \phi(a,b')$.
\end{proof}
\begin{fact} \label{rational-ff}
  Let $W$ and $V$ be varieties over $K$, with $V$ integral.  Then
  $W(K(V))$ can be identified with the set of rational maps $V
  \dashrightarrow W$ over $K$.  Under this identification, $W(K)$
  corresponds to the constant maps.
\end{fact}
We now characterize the e.c.\@ models of $T_{C,K_0}$.
\begin{proposition} \label{prop:axiom}
  $K \models T_{C,K_0}$ is e.c.\@ if and only if the following two
  conditions hold:
  \begin{itemize}
  \item[$(\mathrm{I}_C)$] If $L/K$ is any proper finite extension,
    then $C(L) \ne \varnothing$.
  \item[$(\mathrm{II}_C)$] For any geometrically integral
    variety $V/K$, either there is a dominant rational map $f : V
    \dashrightarrow C$ over $K$ or $V(K)$ is Zariski dense in $V$.
  \end{itemize}
\end{proposition}
\begin{proof}
    First suppose $K$ is e.c.  For ($\mathrm{I}_C$), if $L/K$ is a
    proper finite extension then $K \not \prec_\exists L$, since $L$ contains a root of some polynomial over $K$ that has no solution in $K$. So $L
    \not \models T_{C,K_0}$, meaning that $C(L) \ne \varnothing$.  For
    ($\mathrm{II}_C$), suppose there is no dominant
    $f:V\dashrightarrow C$ over $K$.  By Fact~\ref{rational-ff}, $C(K(V))=C(K)=
    \varnothing$, and so $K(V) \models T_{C,K_0}$.  Then $K$ is e.c.\@ in
    $K(V)$, which implies that $V(K)$ is Zariski dense in $V$ by
    Fact~\ref{fact:ec}.

    Conversely, suppose that ($\mathrm{I}_C$) and
    ($\mathrm{II}_C$) hold.  Let $L \models T_{C,K_0}$ be an
    extension of $K$.  We need to show that $K$ is e.c.\@ in $L$.  By
    Lemma~\ref{lem:fg}, we may assume that $L$ is finitely generated
    over $K$.  Then $L=K(V)$ for some integral variety $V$ over $K$.
    Note that $V$ must be geometrically integral, because otherwise
    $L$ contains a proper finite extension $F$ of $K$, and then $C(L)
    \supseteq C(F) \supsetneq \varnothing$ by ($\mathrm{I}_C$),
    contradicting $L \models T_{C,K_0}$.  Moreover, there is no dominant
    rational map $f : V \dashrightarrow C$ over $K$ or else $C(L) =
    C(V(K)) \ne \varnothing$ by Fact~\ref{rational-ff}.  Therefore
    $V(K)$ is Zariski dense in $V$ by ($\mathrm{II}_C$).  By
    Fact~\ref{fact:ec}, $K$ is e.c.\@ in $K(V) = L$.
\end{proof}
To complete the proof that the model companion $C\XF$ exists (Theorem~\ref{thm:cxf-exist}), we must show that the conditions
($\mathrm{I}_C$) and ($\mathrm{II}_C$) are first-order (see Fact~\ref{mc-fact}(2)).  Condition
($\mathrm{I}_C$) is easy.  In any field $K$, there is a natural way to
interpret the family of degree $n$ field extensions, for example by
considering all possibilities for the structure coefficients.  Then it
is straightforward to express the statement that $C(L) \ne
\varnothing$ for every $L/K$ with $[L : K] = n$.

In contrast, condition ($\mathrm{II}_C$) will take more work.  We
describe the proof in the remainder of this section, leaving one
detail for the following section.  For notational simplicity, we
assume that $K_0 = \Qq$.  Otherwise, $\clring$ should be replaced with
$\clring(K_0)$ as appropriate.

First we introduce some machinery which isn't strictly necessary, but
is convenient for transfering statements between $K^{\alg}$ and $K$.
\begin{definition} \label{lp-def}
  Let $\mathcal{L}_P$ be $\clring$ plus a unary predicate symbol $P$,
  and let $T_P$ be the theory of pairs $(M,K)$ where $M$ is an
  algebraically closed field, $K$ is a subfield, and $M$ is \emph{not}
  a finite extension of $K$.
\end{definition}
For any field $K$, there is an extension $M$ with $(M,K) \models T_P$.
Usually we can take $M = K^\alg$, and in general we can take $M =
K(t)^\alg$, for example.
\begin{fact} \label{tp-fact}
  If $(M_i,K_i) \models T_P$ for $i = 1, 2$, then $(M_1,K_1) \equiv
  (M_2,K_2)$ if and only if $K_1 \equiv K_2$ as pure fields.
\end{fact}
Fact~\ref{tp-fact} is an easy exercise in model theory.\footnote{Fact~\ref{tp-fact} is equivalent to Corollary~\ref{tp-cor}, which can be formulated as a statement about the natural numbers.  In particular, it is absolute.  Consequently, we may move to the constructible universe $L$.  Then we may assume that $(M_1,K_1)$ and $(M_2,K_2)$ are saturated of size
$\kappa$ and that $K_1 \cong K_2$.  With some work, one can show that
$\trdg(M_i/K_i) = \kappa$.  This determines $M_i$ up to
isomorphism as an extension of $K_1$, and so $K_1 \cong K_2$ implies
$(M_1,K_1) \cong (M_2,K_2)$.}  By Fact~\ref{tp-fact} and compactness, the restriction
map $\Th(M,K) \mapsto \Th(K)$ is a homeomorphism from the space of
completions of $T_P$ to the space of completions of the theory of
fields.
\begin{corollary} \label{tp-cor}
  If $\phi$ is a sentence in the language of $T_P$, then there is a
  $\clring$-sentence $\phi'$ such that $(M,K) \models \phi \iff K
  \models \phi'$.
\end{corollary}
Consequently, to express $(\mathrm{II}_C)$ via a set of
$\clring$-sentences, it suffices to express it as a set of
$\mathcal{L}_P$-sentences about the pair $(M,K) \models T_P$.  The
advantage of $(M,K)$ is that we can regard varieties over $K$ as
definable sets in $M$, in the following way.

If $M$ is an algebraically closed field, we can identify $\Pp^n(M)$
with a definable set in $M$.  Then we can identify quasiprojective
varieties $X \subseteq \Pp^n$ with the corresponding definable sets
$X(M) \subseteq \Pp^n(M)$.  The map $X \mapsto X(M)$ is faithful
because we assume that ``varieties'' are reduced.  If $K$ is a
subfield of $M$, then we can unambiguously talk about a variety $V
\subseteq \Pp^n$ being ``defined over $K$''---the model-theoretic
and algebro-geometric notions agree because we are working in characteristic
$0$.  Another consequence of the characteristic $0$ assumption is that two irreducible
varieties $V, W$ are birationally equivalent over $K$ if and only if
there is a $K$-definable generic bijection between $V$ and $W$, or
more precisely, a $K$-definable bijection $V_0 \to W_0$ for some
definable subsets $V_0 \subseteq V$ and $W_0 \subseteq W$ with
$\dim(V_0) = \dim(V)$ and $\dim(W_0) = \dim(W)$.  This is because
definable functions are piecewise algebraic in an algebraically closed field of characteristic $0$.

Next, we need a handful of facts about definability and complexity in
ACF.  Most of these facts are well-known or easy exercises, but one of
them takes more work and will be proved in the following section.
Throughout, work in an algebraically closed field $M$.

If $X \subseteq \Pp^n$ is a projective variety, we let $\pdeg(X)$
denote the projective degree of $X$.
\begin{fact} \label{pdeg-def}
  If $\{V_a\}_{a \in Y}$ is a definable family of projective
  varieties, then $\pdeg(V_a)$ depends definably on $a$.
\end{fact}
This holds because we can calculate $\pdeg(V)$ for $d$-dimensional $V$
by counting the number of intersections of $V$ with a generic
$(n-d)$-dimensional linear subspace of $\Pp^n$.  Model
theoretically, the definability follows using the definability of
types applied to the generic type of the Grassmannian of
$(n-d)$-dimensional linear subspaces of $\Pp^n$.

For a given number $d,r$ and projective space $\Pp^n$, it is a classical theorem of Chow stating that the Chow variety $\mathrm{Gr}(d,r,n)$ exists, parameterizing effective cycles of $\Pp^n$ with degree $d$ and dimension $r$.  See \cite[Chapter 21]{HarrisAG} for a proof. From this, we extract the following:
\begin{fact}[Existence of Chow varieties] \label{chow-varieties}
  Let $M$ be an algebraically closed field and $n, m$ be integers.
  There is a 0-definable family $\{V_a\}_{a \in Y}$ containing every
  projective variety $X \subseteq \Pp^n$ with $\pdeg(X) \le m$.
  Moreover, if $K$ is a subfield, then $\{V_a : a \in Y(K)\}$
  contains every projective variety $X \subseteq \Pp^n$ defined over
  $K$ with $\pdeg(X) \le m$.
\end{fact}
The part about $K$-definability is clear from the proof, or can be
obtained by using elimination of imaginaries to ensure that $a$ is a
code for $V_a$.
\begin{fact} \label{degree-bound}
  There is a computable function $B(x,y,z)$ with the following property.
  Let $X \subseteq \Pp^n$ and $C \subseteq \Pp^m$ be smooth projective
  integral varieties over $M = M^{\alg}$, with $C$ a curve with genus $g_C\geq 2$.  Regard $X
  \times C \subseteq \Pp^n \times \Pp^m$ as a subvariety of
  $\Pp^{(n+1)(m+1)-1}$ via the Segre embedding.  Let $f : X
  \dashrightarrow C$ be a dominant rational map.  Let $\Gamma$ be the
  Zariski closure of $\Gamma_f$ in $X \times C \subseteq
  \Pp^{(n+1)(m+1)-1}$ via the Segre embedding.  Then the projective degree $\pdeg(\Gamma)$
  (calculated inside $\Pp^{(n+1)(m+1)-1}$) is at most
  $B(\pdeg(X),\pdeg(C),\dim(X))$.
\end{fact}
We will prove Fact~\ref{degree-bound} in the next section.
\begin{lemma} \label{colossal-waste-of-time}
    Let $X$ and $Y$ be irreducible projective varieties.  A closed subvariety $\Gamma \subseteq X \times Y$ is the Zariski closure of the graph of a rational map $f : X \dashrightarrow Y$ if and only if the following conditions all hold:
    \begin{itemize}
        \item $\Gamma$ is irreducible.
        \item For generically-many $a \in V$, the projection $\Gamma \to V$ has fiber size 1.  In other words, the definable set \[
        \{ a \in V(M) : |\Gamma \cap (\{a\} \times Y)| = 1\}\] is generic in $V$.
    \end{itemize}
    When these conditions hold, the rational map is dominant if and only if $\Gamma$ projects onto $Y$.
\end{lemma}
\begin{proof}
    First suppose $\Gamma$ is the Zariski closure of the graph $\Gamma_f$ of a morphism $f : U \to Y$ for some Zariski dense open set $U \subseteq X$.  Since $\Gamma_f$ is Zariski closed in the Zariski open subset $U \times Y \subseteq X \times Y$, we have $\Gamma \cap (U \times Y) = \Gamma_f$, and the second condition holds.  If the first condition fails, then $\Gamma$ is a union of two smaller Zariski closed sets $A \cup B$.  Restricting to $U \times Y$, we see that \[\Gamma_f = \Gamma \cap (U \times Y) = (A \cap (U \times Y)) \cup (B \cap (U \times Y)).\]  Since $\Gamma_f$ is irreducible as a variety (being isomorphic to $U$), one of $A$ or $B$ must contain $\Gamma_f$, and then contain its Zariski closure $\Gamma$, a contradiction.

    Conversely, suppose the two conditions hold.  Let $D$ be the definable set of $a \in X$ such that there is a unique $b \in Y$ such that $(a,b) \in \Gamma$.  By assumption, $D$ is generic.  Let $f_0 : D \to Y$ be the function whose graph is $\Gamma \cap (D \times Y)$.  Because we are in characteristic 0, $f_0$ is generically given by a rational function.  So there is a Zariski dense open $U \subseteq X$ such that $U \subseteq D$ and $f_0 \restriction U$ is some morphism $f : U \to Y$.  Then $\Gamma \cap (U \times Y) = \Gamma_f$.  It follows that $\Gamma$ is contained in the union of the following closed sets: (1) $(X \setminus U) \times Y$ and (2) the Zariski closure of the graph $\Gamma_f$.  Since $\Gamma$ is irreducible, it must be contained in one set or the other.  It can't be contained in $(X \setminus U) \times Y$, since it intersects the complementary set $U \setminus Y$.  Instead, $\Gamma$ must be contained in the Zariski closure of $\Gamma_f$.  But $\Gamma$ is Zariski closed and contains $\Gamma_f$, so $\Gamma$ must equal the Zariski closure of $\Gamma_f$.

    For the final point, if $f$ is non-dominant, then its image is contained in some closed subvariety $W \subseteq Y$.  As $\Gamma$ is the Zariski closure of the graph, it is contained in $X \times W$, so it does not project onto $Y$.  Conversely, suppose $\Gamma$ does not project onto $Y$.  Since the varieties are projective, hence complete, the image of the projection is some closed proper subvariety $W \subsetneq Y$.  The graph of $f$ is contained in $X \times W$, so $f$ maps $X$ into $W$, and $f$ is non-dominant.
\end{proof}
\begin{fact} \label{irreducibility-def}
  Let $\{V_a\}_{a \in Y}$ be a definable family of subsets of
  $\Pp^n(M)$.  Let $Y'$ be the set of $a \in Y$ such that $V_a$ is an
  irreducible closed subvariety of $\Pp^n$.  Then $Y'$ is definable.
\end{fact}
See for example~\cite[Appendix]{diff-chow} for a model-theoretic proof or~\cite[Tag~0554]{stacks-project} for a geometric treatment. 

The celebrated theorem of Hironaka states that every variety admits
resolution of singularities in characteristic
$0$~\cite{Hironaka-resolution}.  Specifically:
\begin{fact}[Hironaka] \label{hironaka}
  If $V$ is a geometrically irreducible variety over $K$, then there
  is a smooth projective irreducible variety $V'$ over $K$, and a
  birational map $V \dashrightarrow V'$ over $K$.
\end{fact}
The following fact essentially says that the class of smooth
projective varieties in $\Pp^n$ is ind-definable, i.e., a small union
of definable families.
\begin{fact} \label{families}
  Let $K$ be a subfield of the algebraically closed field $M$.  Let
  $V \subseteq \Pp^n$ be a smooth projective subvariety defined over
  $K$.  Then there is a 0-definable family $\{V_a\}_{a \in Y}$ such
  that each $V_a$ is a smooth projective subvariety of $\Pp^n$, and $V
  = V_a$ for some $a \in Y(K)$.
\end{fact}
Fact~\ref{families} is easy to see using the Jacobian criterion of
smoothness.

Putting everything
together, we can express axiom $(\mathrm{II}_C)$ for $K$ by saying the following
about the pair $(M,K) \models T_P$.
\begin{quote}
  Let $n$ be an integer and $\phi(x,y)$ be an $\clring$-formula.
  Let $\{V_a\}_{a \in Y}$ be a 0-definable family in $M$ such that
  each $V_a$ is a smooth projective variety $V_a \subseteq \Pp^n$.
  For each $a \in Y(K)$ and each $b$ in $K$, if $V_a(M)$ is
  irreducible and $\phi(M,b) \cap V_a$ has the same
  dimension as $V_a$, at least one of the following happens:
  \begin{enumerate}
  \item There is a $K$-point in $\phi(M,b) \cap V_a$.
  \item There is a $K$-definable irreducible projective variety $\Gamma \subseteq
    V_a \times \tilde{C}$ with $\pdeg(\Gamma) \le B(\pdeg(V_a),\pdeg(\tilde{C}),\dim(V_a))$ such that $\Gamma(M)$
    is the Zariski closure of the graph of a dominant rational map
    from $V_a(M)$ to $\tilde{C}(M)$.
  \end{enumerate}
  In (2), $\tilde{C} \subseteq \Pp^m$ is the projective completion of
  $C$, and $B(x,y,z)$ is the computable function from
  Fact~\ref{degree-bound}.  The projective degree of $\Gamma$ is computed with respect to the Segre embedding.
\end{quote}
This is an axiom schema, with one axiom for each combination of $n$,
$\{V_a\}_{a \in Y}$, and $\phi(x,y)$.  Irreducibility is definable by
Fact~\ref{irreducibility-def}.  The projective degree $\pdeg(V_a)$ is
definable by Fact~\ref{pdeg-def}.  Definability of dimension is well-known---it holds in any strongly minimal theory.  We can quantify over the
possibilities for $\Gamma$ using Chow varieties
(Fact~\ref{chow-varieties}).  We can express the condition ``$\Gamma(M)$ is the Zariski closure of the
graph of a dominant rational map\ldots'' using Lemma~\ref{colossal-waste-of-time}, again using definability of irreducibility and definability of dimension.  Thus, everything is genuinely
first-order. 

Moreover, the axiom schema is equivalent to $(\mathrm{II}_C)$.  First suppose Condition $(\mathrm{II}_C)$ holds.  In the setting of the axiom schema, $V_a$ is a geometrically integral $K$-variety, and so there are two cases by condition $(\mathrm{II}_C)$:
\begin{itemize}
  \item $V_a(K)$ is Zariski dense in $V_a$, in the sense that $V_a(K)$ intersects any non-empty $K$-definable Zariski open set $U \subseteq V_a$.  Since $V_a$ is irreducible, it has a unique generic type, so the definable set $V_a \setminus \phi(M,b)$ has lower dimension than $V_a$.  So does its Zariski closure $\overline{V_a \setminus \phi(M,b)}$.  (Any definable set has the same dimension as its Zariski closure.)  Take $U$ to be the complementary open set $V_a \setminus \overline{V_a \setminus \phi(M,b)}$.  Then $U$ is non-empty.  It is $K$-definable because $V_a$ and $\phi(M,b)$ are.  By the Zariski density assumption, $U$ intersects $V_a(K)$.  But $U \subseteq \phi(M,b)$, so condition (1) of the axiom schema holds.
    \item There is a dominant rational map $f : V_a \dashrightarrow C$ over $K$.  Take $\Gamma$ to be the Zariski closure in $V_a \times \tilde{C}$ of the graph $\Gamma_f$.  By Fact~\ref{degree-bound}, $\pdeg(\Gamma) \le B(\pdeg(V_a),\pdeg(\tilde{C}),\dim(V_a))$.  Then condition (2) in the axiom schema holds.
\end{itemize}
Conversely, suppose the axiom schema holds.  We verify condition $(\mathrm{II}_C)$.  Suppose $V$ is a geometrically irreducible variety over $K$ without any $K$-definable dominant rational map to $C$.  We must show that $V(K)$ is Zariski dense in $V$.  By Hironaka's theorem (Fact~\ref{hironaka}), there is a $K$-definable smooth projective
variety $V' \subseteq \Pp^n$ that is birationally equivalent to $V$
over $K$.  We may replace $V$ with $V'$---this does not affect the Zariski density of $K$-rational points or the (non-)existence of a $K$-definable dominant rational map to $C$.  So we may assume that $V$ is smooth and projective.  By Fact~\ref{families}, $V$ belongs to some 0-definable
family $\{V_a\}_{a \in Y}$ of smooth projective varieties in $\Pp^n$.  Fix $a \in Y(K)$ such that $V = V_a$.  Let $U$ be a non-empty $K$-definable Zariski open subset of $V_a$.  We can write $U$ as $\phi(M,b)$ for some formula $\phi$ and parameter $b$ from $K$.  Because $V_a$ is irreducible and $U$ is a non-empty Zariski open subset, $\dim(U) = \dim(V_a)$.  Then the axiom schema applies.  Since there are no $K$-definable dominant rational maps from $V_a$ to $C$, we instead get a $K$-point in $\phi(M,b) \cap V_a =  U \cap V_a = U$.  As $U$ was an arbitrary $K$-definable Zariski open subset of $V$, we see that $V(K)$ is Zariski dense in $V$.
% Condition $(\mathrm{II}_C)$ implies the schema---this is obvious
% except for the bound on $\pdeg(\Gamma)$, which comes from
% Fact~\ref{degree-bound}.  In the other direction, suppose $V_0$ is a
% geometrically irreducible variety over $K$ without any $K$-definable
% dominant rational map to $C$.  By Hironaka's theorem
% (Fact~\ref{hironaka}), there is a $K$-definable smooth projective
% variety $V \subseteq \Pp^n$ that is birationally equivalent to $V_0$
% over $K$.  By Fact~\ref{families}, $V$ belongs to some 0-definable
% family $\{V_a\}_{a \in Y}$ of smooth projective varieties in $\Pp^n$.
% Then the axiom schema shows that any Zariski dense $K$-definable
% subset $V \cap \phi(M,b)$ contains a $K$-point, which is condition
% ($\mathrm{II}_C$).

This completes the proof of Theorem~\ref{thm:cxf-exist}, and so the
model companion $C\XF$ exists.  From Proposition~\ref{prop:axiom}, we
see that a model of $C\XF$ is a field $K/K_0$ satisfying the following
conditions:
\begin{itemize}
\item[$(\mathrm{I}_C)$] If $L/K$ is any proper finite extension,
    then $C(L) \ne \varnothing$.
\item[$(\mathrm{II}_C)$] For any geometrically integral
    variety $V/K$, either there is a dominant rational map $f : V
    \dashrightarrow C$ over $K$ or $V(K)$ is Zariski dense in $V$.
\item[$(\mathrm{III}_C)$] $C(K) = \varnothing$
\end{itemize}
We will see later in Section~\ref{sec:other-properties} that Axiom
$(\mathrm{I}_C)$ can be omitted, as it follows from the other two
axioms.

\begin{remark} \label{rem:computable}
  The axioms of $C\XF$ are computable.  This can be seen by tracing
  through the proof.  Here are some relevant observations:
  \begin{enumerate}
  \item Condition $(\mathrm{I}_C)$ is easy to axiomatize, and the axioms
    are computable by inspection.
  \item The conversion between $\mathcal{L}_P$-sentences and
    $\clring$-sentences is computable because $T_P$ is computably
    axiomatized.
  \item In Fact~\ref{chow-varieties}, the formulas defining
    $\{V_a\}_{a \in Y}$ depend computably on $n$ and $m$, by
    inspecting the construction.
  \item Similarly, the definable sets given by Fact~\ref{pdeg-def} and
    \ref{irreducibility-def} depend in a computable way on the
    formulas defining $\{V_a\}_{a \in Y}$, and this can be seen by
    inspecting the proofs.
  \end{enumerate}
\end{remark}
\begin{question}
Recall that when axiomatizing large fields or PAC fields, one can reduce to the case of curves.  For example, a field $K$ is PAC if $C(K)$ is non-empty for any geometrically integral curve $C$ over $K$.  This condition implies the stronger statement that $V(K)$ is non-empty for any geometrically integral variety $V$ over $K$, by an argument using Bertini's theorem.

Does the same hold for $C\XF$?  More specifically, in axiom (II$_C$), can we reduce to the case where $V$ is a curve?  We could not see how to apply the usual Bertini arguments here.
\end{question}

\section{Complexity bounds of rational morphisms}\label{sec:inter}
In this section, we prove Fact~\ref{degree-bound}, bounding the
projective degree of the Zariski closure of the graph of a dominant
rational map $X \dashrightarrow C$ in terms of the projective degrees
of $X$ and $C$.  Throughout the section, we work over an algebraically closed field $K$ of characteristic 0. 

Before giving the proof, we recall some basic intersection theory.
Unless noted otherwise, the facts here are stated in \cite[Section
  A.1]{Hartshorne-AG}, and proved in \cite{Fulton-intersection}.
\begin{enumerate}
\item \emph{(Chow ring.)} To any $d$-dimensional smooth projective variety $X$ over $K$,
  there is associated a commutative graded unital ring $A(X) = \bigoplus_{i =
    0}^d A^i(X)$ called the \emph{Chow ring}.  Multplication is
  written $\alpha.\beta$ rather than $\alpha \cdot \beta$.
\item \emph{(Classes of subvarieties.)} If $V$ is an irreducible subvariety of $X$ of codimension $i$,
  then there is a corresponding element $[V] \in A^i(X)$, and these
  elements generate $A^i(X)$ as a group.  If $V$ is a reducible
  subvariety with irreducible components $V_1, \ldots, V_k$, then we
  define $[V] = \sum_{i = 1}^k [V_i]$.  The element $[X] \in A^0(X)$
  is the identity element of the ring.
\item \emph{(Classes of divisors.)} If $D$ is a divisor\footnote{Weil or Cartier---the two are
equivalent by smoothness of $X$.} $D = \sum_{i = 1}^k n_iV_i$, then we
  define $[D] = \sum_{i = 1}^k n_i[V_i] \in A^1(X)$.  Then $[D] =
  [D']$ if and only if $D$ is linearly equivalent to $D'$, and so $A^1(X)$ can be
  identified with the Picard group of $X$.
\item \emph{(Pullback maps.)} If $f : X \to Y$ is a morphism of smooth
  projective varieties, there is a graded ring homomorphism $f^* :
  A(Y) \to A(X)$ called the \emph{pullback}.  This defines a
  contravariant functor, so that $(f \circ g)^* = g^* \circ f^*$ and
  $\id^* = \id$.  The pullback map
  $A^1(Y) \to A^1(X)$ agrees with the usual pullback map on Picard
  groups.  When $\pi : X \times Y \to X$ is a projection, the
  pullback map sends $[V] \in A(X)$ to $[V \times Y] \in A(X \times
  Y)$.
\item \emph{(Pushforward maps.)} If $f : X \to Y$ is a morphism of smooth projective varieties,
  there is a graded group homomorphism $f_* : A(X) \to A(Y)$ shifting
  degrees by $\dim(Y) - \dim(X)$, called the \emph{pushforward}.  This
  defines a covariant functor, so that $(f \circ g)_* = f_* \circ g_*$
  and $\id_* = \id$.  The pushforward map sends $[V]$ to
  $\deg(V/f(V))[f(V)]$, where $\deg(V/f(V))$ is the degree of the
  function field extension $[K(V) : K(f(V))]$, or 0 if the extension
  is infinite (i.e., $\dim(f(V)) < \dim(V)$).
\item \emph{(Degree map.)} If $X$ has dimension $d$, then the
  pushforward along $X \to \Spec K$ gives a linear map
  \begin{gather*}
    A^d(X) \to \Zz \\
    \alpha \mapsto \int_X \alpha
  \end{gather*}
  characterized by the fact that $\int_X [p] = 1$ for $p$ a point.
\item \emph{(Projection formula.)} If $f : X \to Y$ is a morphism, $\alpha \in A(X)$ and $\beta \in
  A(Y)$, then the \emph{projection formula} holds:
  \begin{equation*}
    f_*(\alpha . f^*(\beta)) = f_*(\alpha).\beta.
  \end{equation*}
  If $\alpha$ and $\beta$ have suitable gradings, this implies
  \begin{equation*}
    \int_X \alpha . f^*(\beta) = \int_Y f_*(\alpha).\beta
  \end{equation*}
  which we also call the projection formula.
\item \emph{(Transverse intersections.)} Let $V_1, V_2, \ldots, V_k$ be varieties intersecting
  transversally, meaning that if $p$ is a generic point of $V_1 \cap
  \cdots \cap V_k$, then $p$ is a smooth point of each $V_i$ and the
  tangent spaces $T_p V_i$ are transverse as vector spaces within $T_p
  X$.  Then $[V_1].[V_2].\cdots.[V_k] = [V_1 \cap \cdots \cap V_k]$.
  See \cite[Proposition~8.2, Example~8.2.1, and
    Remark~8.3]{Fulton-intersection}.
\item \emph{(Projective degree.)} If $X$ is a $d$-dimensional irreducible closed subvariety of
  $\Pp^n$, and $H_1, \ldots, H_d$ are generic hyperplanes in $\Pp^n$,
  then $X, H_1, \ldots, H_d$ intersect transversally, and so
  \begin{equation*}
    [X].[H_1].\ldots.[H_d] = [X \cap H_1 \cap \cdots \cap H_d].
  \end{equation*}
  Therefore $\int_{\Pp^n} [X].[H_1].\ldots.[H_d]$ is the number of
  points in the intersection $X \cap H_1 \cap \cdots \cap H_d$, which
  is the projective degree $\pdeg(X)$.  The $H_i$ are all linearly
  equivalent to each other, so $[H_1] = [H_2] = \cdots = [H_n] =
  \alpha$ for some element $\alpha \in A^1(\Pp^n)$.  Therefore
  \begin{equation*}
    \int_{\Pp^n} \alpha^d.[X] = \pdeg(X).
  \end{equation*}
\end{enumerate}
Now we return to the setting of Fact~\ref{degree-bound}.  Suppose that $X \subseteq \Pp^n$ is an integral smooth projective variety,
$C \subseteq \Pp^m$ is a projective curve of genus at least 2, and $f
: X \dashrightarrow C$ is a dominant rational morphism.  Let $\Gamma$
be the closure of the graph of $f$.  Note that
\begin{equation*}
  \Gamma \subseteq X \times C \subseteq \Pp^n \times \Pp^m \subseteq \Pp^{(n+1)(m+1)-1}
\end{equation*}
where the final inclusion is the Segre embedding.  Our goal is to
bound $\pdeg(\Gamma)$ in terms of $\pdeg(X)$ and $\pdeg(C)$.  Let
$\alpha \in A^1(\Pp^n)$ be the class of a hyperplane, and define
$\beta \in A^1(\Pp^m)$ and $\gamma \in A^1(\Pp^{(n+1)(m+1)-1})$
similarly.  From the definition of the Segre embedding, one can see
that $\gamma$ pulls back to $\pi_1^*(\alpha) + \pi_2^*(\beta)$ in
$A^1(\Pp^n \times \Pp^m)$.  If $V$ is one of the varieties in the
following diagram
\begin{equation*}
  \xymatrix{
    & X \times C \ar[dl] \ar[d] \ar[dr] & \\ X \ar[d] & \Pp^n \times \Pp^m \ar[dl] \ar[d] \ar[dr] & C \ar[d] \\ \Pp^n & \Pp^{(n+1)(m+1)-1} & \Pp^m,
  }
\end{equation*}
let $\alpha_V$ or $\beta_V$ or $\gamma_V$ denote the pullback of
$\alpha$, $\beta$, or $\gamma$ along the arrows.  Abusing notation, we
omit the subscript when $V$ is clear from context.  Then $\alpha +
\beta = \gamma$ holds in $A(\Pp^n \times \Pp^m)$ and $A(X \times C)$.

Let $d = \dim(X) = \dim(\Gamma)$.  Take $H_1, \ldots, H_{d-1}$ generic
hyperplanes in $\Pp^n$.  By Bertini's theorem, they intersect $X$
transversally and the intersection $D := H_1 \cap \cdots \cap H_{d-1}
\cap X$ is a smooth connected curve.  If $\Omega \subseteq X$ denotes
the maximal domain of definition of $f$, then $X \setminus \Omega$ has dimension at most
$d-2$ by the valuative criterion of properness~(see \cite[Lemma 3.2]{milne_abelian} for example), so $H_1 \cap \cdots \cap H_{d-1}$ does not intersect $X
\setminus \Omega$.  Thus $D \subseteq \Omega$, and $f$ is defined on
$D$.  Let $\Gamma_D$ denote the restriction of the graph of $f$ to
$D$.  Then $\Gamma_D$ is a closed subvariety of $\Gamma$, isomorphic
to $D$.

Note that $\alpha = [H_i]$ in $A(X)$, so $\alpha = [H_i \times C]$ in
$A(X \times C)$.  The intersection
\begin{equation*}
  (H_1 \times C) \cap (H_2 \times C) \cap \cdots \cap (H_{d-1} \times
  C) \cap \Gamma
\end{equation*}
is transverse and equal to $\Gamma_D$, so
\begin{equation*}
  \alpha^{d-1}[\Gamma] = [\Gamma_D] \text{ in } A(X \times C).
\end{equation*}
Similarly,
\begin{gather*}
  \alpha^{d-1}[X] = [D] \text{ in $A(X)$ or $A(\Pp^n)$}.
\end{gather*}
Now we calculate
\begin{equation*}
  \pdeg(\Gamma) = \int_{\Pp^{(n+1)(m+1)-1}} \gamma^d [\Gamma] =
  \int_{X \times C} \gamma^d [\Gamma] = \int_{X \times C} (\alpha +
  \beta)^d [\Gamma] = \sum_{i = 0}^d \binom{d}{i} \int_{X \times C}
  \alpha^{d-i}\beta^i [\Gamma],
\end{equation*}
where the second equality holds by the projection formula, as
$[\Gamma] \in A(X \times C)$ pushes forward to $[\Gamma] \in
\Pp^{(n+1)(m+1)-1}$.  Meanwhile, $\beta^2 = 0$ in $A(C)$, because
$\beta$ has grading 1 and $A(C) = A^0(C) \oplus A^1(C)$.  Then
$\beta^2 = 0$ in $A(X \times C)$ as well, and so
\begin{equation*}
  \pdeg(\Gamma) = \int_{X \times C} \alpha^d [\Gamma] + d \int_{X
    \times C} \alpha^{d-1} \beta [\Gamma].
\end{equation*}
The first summand can be rewritten using the projection formula as
\begin{equation*}
  \int_{X \times C} \alpha^d [\Gamma] = \int_{\Pp^n} \alpha^d [X] = \pdeg(X),
\end{equation*}
because $[\Gamma] \in A(X \times C)$ pushes forward to $[X] \in
A(\Pp^n)$.

For the second summand, 
\begin{equation*}
  \int_{X \times C} \alpha^{d-1}[\Gamma]\beta = \int_{X \times C}
      [\Gamma_D]\beta = \deg(D/C) \int_{\Pp^m} \beta [C] = \deg(D/C)\pdeg(C),
\end{equation*}
because $[\Gamma_D]$ pushes forward to $\deg(D/C)[C]$ in $\Pp^m$.
Putting everything together,
\begin{equation*}
  \pdeg(\Gamma) = \pdeg(X) + d \deg(D/C) \pdeg(C).
\end{equation*}
Now let $g_C$ and $g_D$ be the genera of $C$ and $D$.  Then
\begin{equation*}
  \deg(D/C) \le \frac{g_D - 1}{g_C - 1} \le g_D - 1
\end{equation*}
by the Riemann-Hurwitz theorem, as $g_C \ge 2$ by assumption.\footnote{In order to apply Riemann-Hurwitz, we need the finite morphism $D \to C$ to be separable, and this holds because $K$ has characteristic 0.  It is precisely this step where the construction of $C\XF$ fails in positive characteristic.}  It
remains to bound $g_D$. To this end, the following bound on $g_D$ can be deduced from~\cite{GLP-bound}, see for example~\cite{glp_bound_note}.
\begin{gather*}
  g_D \le \pdeg(D)^2 - 2\pdeg(D) + 1.
\end{gather*}
Finally, $\pdeg(D) = \int_{\Pp^n} \alpha [D] = \int_{\Pp^n} \alpha (\alpha^{d-1}[X]) = \pdeg(X)$.  Thus everything can be bounded in terms of $\pdeg(X)$ and $\pdeg(C)$. Explicitly,
\[
\pdeg(\Gamma)\leq \pdeg(X)+d(\pdeg(X)^2-2\pdeg(X))\pdeg(C).
\]
This completes the proof of Fact~\ref{degree-bound}.

We thank Ofer Gabber for suggesting some of the arguments used in this section. The second author would like to thank Antoine Ducros and Micha\l{} Szachniewicz for helpful discussions on intersection theory.

\section{Properties of $C\XF$} \label{sec:other-properties}
We fix $C$ a curve of genus at least 2 over $K_0$ with $C(K_0) =
\varnothing$.
Slightly abusing notation, let $(C\XF_{K_0})_\forall$ be the theory $T_{C,K_0}$ of
$K_0$-fields that avoid $C$, meaning that $C(K) = \varnothing$.  Let $C\mathrm{XF}_{K_0}$ be the model companion of $(C\XF_{K_0})_\forall$, which exists by Theorem~\ref{thm:cxf-exist}.
  We will usually omit the subscript $K_0$ when it is clear from context, writing $C\XF$ and $C\XF_\forall$.
\begin{theorem} \label{large-characterize}
  Let $\tilde{C}$ be the projective completion of $C$.  If $K \models
  C\XF$, then $K$ is large if and only if $\tilde{C}(K) =
  \varnothing$.
\end{theorem}
\begin{proof}
  As $\tilde{C}$ contains only finitely more points than $C$,
  $\tilde{C}(K)$ is certainly finite.  If $\tilde{C}(K)$ is non-empty,
  then $\tilde{C}$ directly shows that $K$ is not large.  Conversely,
  suppose that $\tilde{C}(K)$ is empty.  We claim that $K$ is large.
  Let $C_1$ be a smooth curve over $K$ with a $K$-point.  We must show that $C_1(K)$ is infinite.
  Break into two cases:
  \begin{enumerate}
  \item There is a $K$-definable dominant rational map $C_1
    \dashrightarrow C$.  Then there is a $K$-definable morphism $C_1
    \to \tilde{C}$.  The image of any point $p \in C_1(K)$ is a point
    in $\tilde{C}(K)$, so $\tilde{C}(K) \ne \varnothing$, a
    contradiction.
  \item There is no $K$-definable dominant rational map $C_1
    \dashrightarrow C$.  By the axioms of $C\XF$, $C_1(K)$ is Zariski
    dense in $C_1$, and so $C_1(K)$ is infinite. \qedhere
  \end{enumerate}
\end{proof}
\begin{corollary} \label{cor:non-large-model-complete}
  There is a non-large model complete field.
\end{corollary}
\begin{proof}
  Take $K_0 = \Qq$.  Take $C_0$ to be the affine Fermat curve $\{(x,y)
  : x^4 + y^4 = 1\}$.  It is well-known that $C_0(\Qq)$ contains only the
  four points $\{(\pm 1, 0), (0, \pm 1)\}$.  Let $C$ be $C_0$ with
  these four points removed.  Then $C(\Qq)$ is empty, so $C\XF$
  exists.  Take $K \models C\XF$.  Then $K$ is model complete.  If
  $\tilde{C}$ is the projective completion of $C$ and $C_0$, then
  $\tilde{C}(K)$ is non-empty because it contains at least the four
  points of $C_0(\Qq)$.  By Theorem~\ref{large-characterize}, $K$ is
  not large.
\end{proof}

Next we investigate other properties of $C\mathrm{XF}$. We need the following lemma.
\begin{lemma} \label{lem:factor}
  Work in $K^{\alg}$.  If $V_1, V_2$ are integral varieties, then any rational map $f : V_1 \times V_2 \dashrightarrow C$ must
  factor through $V_1$ or $V_2$, in the sense that it arises (up to equivalence of rational maps) from a rational map $V_1 \dashrightarrow C$ or from a rational map $V_2 \dashrightarrow C$.
\end{lemma}
\begin{proof}
In the proof, we implicitly use the well-known fact that if $V$ is a variety, then every generic definable set $D \subseteq V$ contains a Zariski dense open subset $U \subseteq V$, and every definable function on a generic subset of $V$ agrees generically with a rational function on $V$ (since we are in characteristic 0).

We may assume that $V_i$'s are smooth and projective since this is a statement about rational maps. Let $\tilde{C}$ denote the smooth projective completion of $C$. Note that $\tilde{C}$ is pure in the sense of~\cite[Definition 3.1]{algebraic_hyperbolicity} by~\cite[Lemma 3.5]{algebraic_hyperbolicity}, because $\tilde{C}$ has positive genus.
% Note that for a variety $V$, a generic definable set $D\subseteq V$ contains a Zariski dense open subset $U\subseteq V$.
For all $a$ on a generic definable subset of $V_1$, the map $f(a,y)$ is a rational map $V_2 \dashrightarrow C$. Thus by~\cite[Lemma 3.2]{algebraic_hyperbolicity}, it corresponds to a morphism $V_2\to \tilde{C}$. There are only finitely many \emph{non-constant} morphisms $V_2 \rightarrow \tilde{C}$ by
  \cite[Corollary 6.3.30]{positive-II} because $\tilde{C}$ is a smooth projective curve with genus at least 2, hence having ample cotangent bundle.  Let $\{g_1,\ldots,g_n\}$
  list all of them.  For $1 \le i \le n$, let $D_i$ be the set of $a
  \in V_1$ such that $f(a,y)$ is equivalent to $g_i(y)$, and let $D_0$
  be the set of $a \in V_1$ such that $f(a,y)$ is generically
  constant.  The sets $D_i$ are definable, and their union is generic in $V_1$.  Therefore, one of them is generic in $V_1$.

  If $D_0$ is generic in $V_1$, let $h : D_0 \to C$ be the
  definable function sending $a$ to the generic value of $f(a,y)$.
  Then $f(x,y) = h(x)$ on a generic definable subset of $V_1 \times V_2$, and $h$ is equivalent to a rational map on $V_1$.
  % This is the where we use the fact the referee wanted us to add.

  If $D_i$ is generic in $V_1$ for some $i \ge 1$, then $f(x,y)
  = g_i(y)$ on a generic definable subset of $V_1 \times V_2$.  In particular, $f(x,y) = g_i(y)$ as rational maps.
  % And here.
\end{proof}
Converting the above statement into algebra, we have the following:
\begin{lemma}\label{lem:regular-tensor}
  If $K$ is a $K_0$-field and $L_1,L_2$ are two regular extensions
  satisfying $C\XF_\forall$, then $\Frac(L_1 \otimes_K L_2) \models
  C\XF_\forall$.
\end{lemma}
\begin{proof}
We may assume that $L_1$ and $L_2$ are finitely generated over $K$, hence $L_i=K(V_i)$ for some geometrically integral $V_i/K$. Since $L_i$ avoids $C$, there is no dominant rational map $V_i\dashrightarrow C$. By Lemma~\ref{lem:factor}, there is no dominant $f:V_1\times V_2\dashrightarrow C$. By Lemma~\ref{rational-ff}, $\mathrm{Frac}(L_1\otimes_K L_2)=K(V_1\times V_2)$ avoids $C$.
\end{proof}
From this, we can easily deduce the following, which enables us to characterize the completions of $C\XF$.
\begin{lemma}\label{lem:partial-elem}
  Let $K$ be a $K_0$-field and $L_1, L_2$ be two regular extensions of
  $K$ satisfying $C\XF$.  Then the identity map from $K$ to $K$ is
  partial elementary between the $L_i$'s.  In other words,
  $L_1\equiv_K L_2$.
\end{lemma}
\begin{proof}
By Lemma~\ref{lem:regular-tensor}, we have at least one model $L \models C\XF_\forall$ such that the following diagram commutes, namely $L = \Frac(L_1 \otimes_K L_2)$.
\[\begin{tikzcd}
	& {L_1} && L \\
	& {} \\
	{} & K && {L_2}
	\arrow[from=1-2, to=1-4]
	\arrow[from=3-2, to=3-4]
	\arrow[from=3-2, to=1-2]
	\arrow[from=3-4, to=1-4]
\end{tikzcd}\] 
Since $C\XF_\forall$ is inductive, by Fact~\ref{mc-fact} we may enlarge $L$ and assume that $L\models C\XF$. Thus the maps $L_i\to L$ are elementary by the model completeness of $C\XF$. Hence the map $\mathrm{id}:K\to K$ is partial elementary from $L_1$ to $L_2$.
\end{proof}
Let $\Abs(M)$ denote $K_0^\alg \cap M$, the relative algebraic closure
of $K_0$ in $M$.
\begin{theorem}\label{completions-1}
Let $L_1,L_2\models C\XF$.  Then $\Abs(L_1) \cong \Abs(L_2)$ iff
$L_1\equiv L_2$.
\end{theorem}
\begin{proof}
  If $\Abs(L_1) \cong \Abs(L_2)$, then $L_1 \equiv L_2$ by
  Lemma~\ref{lem:partial-elem}.  Conversely, if $L_1 \equiv L_2$ then there are elementary embeddings $L_1 \to M$ and $L_2 \to M$ for some structure $M$, by \cite[Lemma~4.11]{Poizat_2000}.  Then $\Abs(L_1) \cong \Abs(M) \cong \Abs(L_2)$.
\end{proof}
So the completions of $C\XF$ correspond to the possibilities for
$\Abs(M)$.  Later, we will see that $\Abs(M)$ can be any $C$-avoiding
algebraic extension of $K_0$.

The following is also immediate from Lemma~\ref{lem:partial-elem}.
\begin{corollary}\label{cor:near-qe}
  $C\XF$ has quantifier elimination if we expand by the solvability
  predicates
  \begin{equation*}
    \mathrm{Sol}_n(x_0,...,x_n)\leftrightarrow \exists y \, x_0+yx_1+...+y^nx_n=0    
  \end{equation*}  
\end{corollary}
\begin{proof}
  This follows formally from Lemma~\ref{lem:partial-elem}.  Suppose
  $M_1, M_2$ are models of $C\XF$, and $A_i$ is a substructure
  (subring) of $M_i$ for $i = \{1,2\}$.  Suppose $f_0 : A_1 \to A_2$
  is an isomorphism in the expanded language.  By a standard
  compactness argument, it suffices to show that $f_0$ is a partial
  elementary map.  First extend $f_0$ to $f_1 : \Frac(A_1) \to
  \Frac(A_2)$.  The resulting map continues to respect the $Sol_n$
  predicates because
  \begin{equation*}
    Sol_n\left(\frac{x_0}{z},\ldots,\frac{x_{n-1}}{z} \right)    \iff Sol_n\left( z^{n-1}x_0,\ldots,z^{n-1-i}x_i,\ldots, x_{n-1}\right).
  \end{equation*}
  Next, let $K_i$ be the relative algebraic closure of $\Frac(A_i)$ in
  $M_i$.  The fact that $f_1$ preserves $Sol_n(-)$ implies that $f_1 :
  \Frac(A_1) \to \Frac(A_2)$ extends to an isomorphism $f_2 : K_1 \to
  K_2$, by Lemma~\ref{rel-closure-sol} below.  Then $f_2$ is a partial
  elementary map by Lemma~\ref{lem:partial-elem}, and so $f_0$ is also
  a partial elementary map.
\end{proof}
\begin{lemma} \label{rel-closure-sol}
  If $M_1, M_2$ are two fields extending a perfect field $K$, and if
  \begin{equation*}
    M_1 \models Sol_n(a_1,\ldots,a_n) \iff M_2 \models Sol_n(a_1,\ldots,a_n)
  \end{equation*}
  for $a_1,\ldots,a_n \in K$, then the relative algebraic closures
  $K^\alg \cap M_1$ and $K^\alg \cap M_2$ are isomorphic over $K$.
\end{lemma}
This is presumably well-known, but we include a proof for lack of a reference.
\begin{proof}
  Let $L_i = K^\alg \cap M_i$, and let $\bar{a}$ be an enumeration of
  $L_1$.  We claim that $\qftp(\bar{a}/K)$ is finitely satisfiable in
  $M_2$.  Indeed, if $(b_1,\ldots,b_n)$ is a finite subtuple of
  $\bar{a}$, then $K(\bar{b}) = K(c)$ for some $c \in L_i$ by the
  primitive element theorem.  If $P(x)$ is the minimal polynomial of
  $c$ over $K$, then $P(x)$ has a root $c'$ in $M_2$ by assumption.
  Then there is an isomorphism $K(\bar{b}) = K(c) \cong K(c')$ sending
  $c$ to $c'$.  If $b_1',\ldots,b_n'$ are the images of
  $b_1,\ldots,b_n$ under this isomorphism, then $\qftp(\bar{b}'/K) =
  \qftp(\bar{b}/K)$, proving the claim.

  By compactness, there is an elementary extension $M_2' \succeq M_2$
  and an infinite tuple $\bar{a}'$ in $M_2'$ realizing $\qftp(\bar{a}/K)$.
  Then $\bar{a}'$ enumerates some algebraic extension $F/K$, so $F
  \subseteq K^{\alg} \cap M_2' = K^{\alg} \cap M_2 = L_2$.  The map
  $a_i \mapsto a'_i$ is an isomorphism $L_1 \cong F$ over $K$.  Thus
  we have an embedding $\iota : L_1 \to L_2$ over $K$.

  By symmetry, we also get an embedding $\iota' : L_2 \to L_1$ over
  $K$.  The composition $\iota' \circ \iota$ is an embedding of $L_1$
  into itself over $K$.  As $L_1/K$ is algebraic, this must be onto.
  (Otherwise, take $\beta \in L_1 \setminus \iota'(\iota(L_1))$ and
  let $P(x)$ be the minimal polynomial of $\beta$ over $K$.  If $S$ is
  the set of roots of $P$ in $L_1$, then the self-embedding $\iota'
  \circ \iota$ maps $S$ into $S \setminus \{\beta\}$, which
  contradicts the pigeonhole principle.)

  Thus $\iota' \circ \iota$ is onto, which implies that $\iota'$ is
  onto, and $\iota'$ is an isomorphism.
\end{proof}
\begin{theorem} \label{acl}
  Suppose $M \models C\XF$ and $K$ is a $K_0$-subfield of $M$.  Then
  $\acl_M(K)$ is the relative algebraic closure of $K$ in $M$.
\end{theorem}
\begin{proof}
  Replacing $K$ with the relative algebraic closure, we may assume
  that $K$ is relatively algebraically closed in $M$.  It suffices to
  show that $\acl_M(K) \subseteq K$.  Let $M'$ be a copy of $M$ over
  $K$.  As in the proof of Lemma~\ref{lem:partial-elem}, build a
  diagram
  \begin{equation*}
    \xymatrix{ K \ar[r] \ar[d] & M \ar[ddr] \ar[d] & \\ M' \ar[drr] \ar[r] & \Frac(M \otimes_K M') \ar[dr] & \\ & & N}
  \end{equation*}
  where $N \models C\XF$.  By model completeness, $M \preceq N$ and
  $M' \preceq N$.  Then $M$ and $M'$ are $\ACF$-independent over $K$
  inside $N$, so $M \cap M' = K$ because $K$ is relatively
  algebraically closed in $M$ and $M'$.

  As algebraic closure is preserved in elementary extensions,
  \begin{equation*}
    \acl_M(K) = \acl_N(K) = \acl_{M'}(K) \subseteq M'
  \end{equation*}
  and so $\acl_M(K) \subseteq M \cap M' = K$.
\end{proof}

\begin{definition}[van den Dries~\cite{lou-dimension}] \label{ab}
		Let $(K,+,\cdot,\ldots)$ be an expansion of a field, and $F$ be a
		subfield.  Then $K$ is \emph{algebraically bounded over $F$} if for
		any formula $\varphi(\bar{x},y)$, there are finitely many polynomials
		$P_1,\ldots,P_m \in F[\bar{x},y]$ such that for any $\bar{a}$, \emph{if}
		$\varphi(\bar{a},K)$ is finite, then $\varphi(\bar{a},K)$ is contained in
		the zero set of $P_i(\bar{a},y)$ for some $i$ such that $P_i(\bar{a},y)$ does not vanish.  Following the convention in \cite{lou-dimension,JK-slim} we say that $K$
		is \emph{algebraically bounded} if it is algebraically bounded over
		$K$.\end{definition}
\begin{corollary} \label{alg-bound}
  Models of $C\XF$ are algebraically bounded.
\end{corollary}
\begin{proof}
  By~\cite[Lemma 2.11]{geom-field}, $K$ is algebraically
bounded if and only if field-theoretic and model-theoretic algebraic
closure \emph{over $K$} agree with each other in elementary extensions
of $K$.  This holds by Theorem~\ref{acl}.
\end{proof}
\begin{corollary} \label{exists-infty}
  $C\XF$ is a geometric theory:
  \begin{enumerate}
  \item $\acl(-)$ satisfies exchange.
  \item $\exists^\infty$ is uniformly eliminated.
  \end{enumerate}
\end{corollary}
\begin{proof}
  (1) holds because $\acl(-)$ agrees with field-theoretic algebraic
  closure, and (2) follows from (1) in any theory of fields
  by~\cite[Theorem 2.5]{geom-field}.
\end{proof}

Next we try to determine all the possible candidates that can serve as $\Abs(K)$ for $K\models C\XF$. We need the following lemma.
\begin{lemma}\label{lem:Weil-res}
Let $K$ be a model of $C\XF_\forall$ and $L$ be a proper finite extension of $K$, let $V$ be a geometrically integral affine variety over $L$, and let $W=\Res_{L/K}(V)$ be the Weil restriction of $V$ to $K$. Then there is no dominant rational map over $K$ from $W$ to $C$.
\end{lemma}
\begin{proof}
By the primitive element theorem, we may assume that $L=K[X]/P(X)$ for some irreducible polynomial $P(X)$ over $K$. For $F$ the splitting field of $P(X)$ over $K$, we have $W_F= \prod_{g\in G} V_F^g$, where $(V_F)^g$ is the conjugate of $(V_F)$ by $g\in G := \Gal(F/K)$. By Lemma~\ref{lem:factor}, we know that a dominant rational map over $K$ will factor through one of the components, hence obtaining a $K$-definable (in ACF) root of $P$, which is a contradiction.
\end{proof}
Let $C\XF'$ be the theory $C\XF$ without the axiom on finite
extensions (axiom (I$_C$)).  Fix some $K \models C\XF'$.  A priori this could be a
more general setting than $K \models C\XF$, but we will see shortly
that $C\XF' \vdash C\XF$.
\begin{theorem} \label{PAC}
  If $K \models C\XF'$ and
  $L/K$ is a proper algebraic extension, then $L$ is PAC.
\end{theorem}
\begin{proof}
  Algebraic extensions of PAC fields are PAC~\cite[Corollary 11.2.5]{field-arithmetic}, so we may assume $L$ is
  a proper finite extension.  Let $V$ be a geometrically integral
  affine variety over $L$.  By Lemma~\ref{lem:Weil-res}, $\Res_{L/K}
  V$ has no $K$-definable dominant rational maps to $C$.  By the
  axioms of $C\XF'$, there is a $K$-point on $\Res_{L/K} V$, or
  equivalently, an $L$-point on $V$.
\end{proof}
\begin{corollary} \label{drop-axiom}
  Suppose $K \models C\XF'$.
  \begin{enumerate}
  \item If $L/K$ is a proper elementary extension, then $C(L) \ne
    \varnothing$.
  \item $K \models C\XF$.
  \end{enumerate}
\end{corollary}
Now that we know $C\XF$ and $C\XF'$ are equivalent, we can use this to
build nice models of $C\XF$:
\begin{theorem} \label{extender}
  Suppose $F \models C\XF_\forall$.  Then there is a regular extension
  $K/F$ with $K \models C\XF$.
\end{theorem}
\begin{proof}
It suffices to find a regular extension $K/F$ with $K \models C\XF'$.
The proof is nearly identical to the well-known proof of Fact~\ref{mc-fact}(1), but we include the proof for completeness.
\begin{claim}
    If $F \models C\XF_\forall$, then there is a regular extension $F'/F$ such that $F' \models C\XF_\forall$, and for any geometrically integral variety $V/F$, either there is a dominant rational map $f : V \dashrightarrow C$ over $F'$ or $V(F')$ is non-empty.
\end{claim}
\begin{claimproof}
    Let $\{V_\alpha\}_{\alpha < \kappa}$ enumerate the geometrically integral varieties $V$ over $F$.  Build an increasing chain of regular extensions $F_\alpha$ as follows:
    \begin{itemize}
        \item If $\alpha = 0$, set $F_\alpha = F$.
        \item If $\alpha$ is a limit ordinal, set $F_\alpha = \bigcup_{\beta < \alpha} F_\beta$.
        \item If $\alpha = \beta+1$ and there is a dominant rational map $V_\beta \dashrightarrow C$ over $F_\beta$, set $F_{\beta+1} = F_\beta$.
        \item If $\alpha = \beta+1$ and there is no dominant rational map $V_\beta \dashrightarrow C$ over $F_\beta$, set $F_{\beta+1} = F_\beta(V_\beta)$.  This will satisfy $C\XF_\forall$ by Fact~\ref{rational-ff} and induction.
    \end{itemize}
    Then take $F' = \bigcup_{\alpha < \kappa} F_\alpha$.
\end{claimproof}
Apply the claim $\omega$ times to build an increasing chain of regular extensions
\begin{equation*}
    F = F_0 \subseteq F_1 \subseteq \cdots
\end{equation*}
such that $F_i \models C\XF_\forall$, and if $V$ is a geometrically integral variety over $F_i$ then either there is a dominant rational map $V \dashrightarrow C$ over $F_{i+1}$, or $V(F_{i+1}) \ne \varnothing$.  Take $K = \bigcup_{i < \omega} F_i$.  Then satisfies axiom $(\mathrm{III}_C)$ ($K \models C\XF_\forall$), and $K$ satisfies the following weaker version of axiom $(\mathrm{II}_C)$:
\begin{quote}
  For any geometrically integral variety $V/K$, either there is a dominant rational map $f : V \dashrightarrow C$ over $K$ or $V(K)$ is non-empty.
\end{quote}
Replacing $V$ with Zariski dense open subvarieties, this implies the full axiom $(\mathrm{II}_C)$.  Thus $K \models C\XF'$
  % Build a tower of regular extensions
  % \begin{equation*}
  %   F_0 \subseteq F_1 \subseteq \cdots \subseteq F_\omega \subseteq
  %   F_{\omega+1} \subseteq \cdots
  % \end{equation*}
  % of length $\kappa$ as follows:
  % \begin{enumerate}
  % \item $F_0 = F$
  % \item $F_{\alpha+1}$ is the function field $F_\alpha(V)$ for some
  %   $F_\alpha$-variety $V_\alpha$ without any $F_\alpha$-definable
  %   dominant rational maps to $C$.
  % \item If $\beta$ is a limit ordinal, then $F_\beta = \bigcup_{\alpha
  %   < \beta} F_\alpha$.
  % \end{enumerate}
  % By induction on $\alpha$, each $F_{\alpha}$ satisfies
  % $C\XF_\forall$.  By choosing $\kappa$ and the $V_\alpha$ correctly,
  % we can easily ensure that the limit $\bigcup_{\alpha < \kappa}
  % F_\alpha$ is a model of $C\XF'$, and therefore a model of $C\XF$.
\end{proof}
\begin{theorem} \label{completions-2}
  If $L/K_0$ is an algebraic extension avoiding $C$, then there is $M
  \models C\XF$ with $\Abs(M) = L$.
\end{theorem}
\begin{proof}
  Take $M/L$ a regular extension with $M \models C\XF$.
\end{proof}
With Theorem~\ref{completions-1}, this characterizes the completions
of $C\XF$:
\begin{itemize}
\item If $K \models C\XF$, then $\Th(K)$ is determined by the
  isomorphism class of $\Abs(K)$ as a $K_0$-field.
\item The possibilities for $\Abs(K)$ are the algebraic extensions of
  $K_0$ avoiding $C$.
\end{itemize}
Next we consider the example from the introduction, and obtain a model
with a decidable theory:
\begin{theorem} \label{thm:fermat-example}
  Suppose $K_0 = \Qq$ and $C$ is the punctured Fermat curve $\{(x,y)
  : x^4 + y^4 = 1, x \ne \pm 1, y \ne \pm 1\}$.
  \begin{enumerate}
  \item There is a model $M \models C\XF$ with $\Abs(M) = \Qq$.
  \item There is a decidable completion of $C\XF$.
  \end{enumerate}
\end{theorem}
\begin{proof}
  Part (1) follows by taking $L = K_0 = \Qq$ in
  Theorem~\ref{completions-2}.  Part (2) follows because $C\XF$ is
  computably axiomatized, and the condition ``$\Abs(M) = \Qq$'' is
  also computably axiomatized. Together, these generate a consistent
  theory by part (1) and a complete theory by
  Theorem~\ref{completions-1}.
\end{proof}
\begin{corollary} \label{cor:decidable-non-large}
  There is an infinite decidable field that is not large.
\end{corollary}
There was no previous known example of this kind, to the best of our knowledge.  This also contradicts the intuition that any ``logically tame'' field should be large.

On the topic of decidability, one can also ask whether the incomplete theory $C\XF$ is decidable---is there an algorithm which takes a sentence $\phi$ and determines whether $C\XF \vdash \phi$?  This turns out to be closely related to the effective Mordell problem in number theory:
\begin{theorem}\label{thm:mordell}
  Let $K_0$ be a number field and $C$ be a curve of genus $\ge 2$ over $K_0$ with no $K_0$-rational points.  Then the following two computational problems are Turing equivalent:
  \begin{enumerate}
      \item Given a sentence $\phi$ in the language of $K_0$-algebras, determine whether $C\XF \vdash \phi$.
      \item Given a finite extension $K/K_0$, determine whether $C(K) = \varnothing$.
  \end{enumerate}
\end{theorem}
\begin{proof}
  Let $\mathcal{L}_{Sol}$ be the language of $K_0$-algebras expanded
  with $Sol_n$ predicates as in Corollary~\ref{cor:near-qe}.  Note the
  following:
  \begin{claim}
    There is a recursive function which assigns to each sentence
    $\phi$ a quantifier-free $\mathcal{L}_{Sol}$-sentence $\phi'$ such
    that $C\XF \vdash \phi \leftrightarrow \phi'$.
  \end{claim}
  Indeed, such a $\phi'$ exists by Corollary~\ref{cor:near-qe}, and we can
  find $\phi'$ recursively by searching logical consequences of $C\XF$
  until we find one of the form $\phi \leftrightarrow \phi'$.
  \begin{claim}
    If $\phi$ is a quantifier-free $\mathcal{L}_{Sol}$-sentence then
    the following are equivalent:
    \begin{itemize}
    \item Some model of $C\XF$ satisfies $\phi$.
    \item There is an algebraic extension $K/K_0$ such that $C(K) =
      \varnothing$ and $K \models \phi$.
    \item There is a finite extension $K/K_0$ such that $C(K) =
      \varnothing$ and $K \models \phi$.
    \end{itemize}
  \end{claim}
  To see this, note that $M \models \phi \iff \Abs(M) \models \phi$,
  and the possibilities for $\Abs(M)$ are characterized by
  Theorem~\ref{completions-2}.  This shows that the first two bullet
  points are equivalent.  The second bullet point implies the third
  bullet point because if we write $K$ as a union of an increasing
  chain of finite extensions $K_0 \subseteq K_1 \subseteq K_2
  \subseteq \cdots$, then $K \models \phi$ implies that $K_i \models
  \phi$ for sufficiently large $i$.  The converse is trivial.

  We now show that problems (1) and (2) are Turing equivalent.  First
  suppose we have an oracle for problem (1).  Let $K$ be a given
  finite extension of $K_0$.  Write $K$ as $K_0(a)$ where $a$ has
  irreducible polynomial $P(x)$.  Then the following are equivalent:
  \begin{itemize}
  \item $C(K) = \varnothing$.
  \item There is a finite extension $L/K$ such that $C(L) = \varnothing$.
  \item There is some finite extension $L/K_0$ such that $C(L) =
    \varnothing$ and $L \models \exists x : P(x) = 0$.
  \end{itemize}
  By the second claim, these are in turn equivalent to the following:
  \begin{itemize}
  \item There is a model $M \models C\XF$ such that $M \models \exists
    x : P(x) = 0$.
  \item $C\XF \not \vdash \neg \exists x : P(x) = 0$.
  \end{itemize}
  We can determine whether the final point holds using the oracle for
  problem (1).

  Conversely, suppose an oracle for problem (2) is given.  Given a
  sentence $\phi$, we must determine whether $C\XF \vdash \phi$.  By the
  first claim, we may assume $\phi$ is a quantifier-free
  $\mathcal{L}_{Sol}$-sentence.  By the second claim, exactly one of
  the following holds:
  \begin{itemize}
  \item $C\XF \vdash \phi$.
  \item There is a finite extension $K/K_0$ such that $C(K) =
    \varnothing$ and $K \models \neg \phi$.
  \end{itemize}
  Run two processes in parallel.  One process searches all logical
  consequences of $C\XF$, looking for a proof of $\phi$.  The other process
  searches all finite extensions $K$ of $K_0$, looking for one such
  that $C(K) = \varnothing$ and $K \models \neg \phi$.  We can test
  whether $C(K) = \varnothing$ using the oracle, and we can test
  whether $K \models \phi$ using the rational roots test.  Eventually
  one process terminates.
\end{proof}
In particular, $C\XF$ would be decidable if there was an effective proof of the Mordell conjecture---an algorithm to find all the $K$-rational points on a curve of genus $\ge 2$ over $K$.
    
Next we aim to show that models of $C\XF$ are Hilbertian, hence unbounded. Recall that a field $K$ is \emph{Hilbertian} if there is $K'$ an elementary extension of $K$ and $t\in K'\setminus K$ such that $K(t)$ is relatively algebraically closed in $K'$. This is equivalent to the usual definition in terms Hilbert's irreduciblity theorem~\cite[Chapter 12.1] {field-arithmetic} by \cite[Proposition 15.1.1]{field-arithmetic}.
\begin{theorem}
  \label{Hilb}
  If $K \models C\XF$, then $K$ is Hilbertian.
\end{theorem}
\begin{proof}
  Build the rational function field $K(t)$.  Note
  that $K(t) \models C\XF$, because there are no $K$-definable
  dominant rational maps $\Aa^1 \dashrightarrow C$.  By
  Theorem~\ref{extender}, there is a regular extension $K'$ of $K(t)$
  such that $K' \models C\XF$.  Then $K' \succeq K$ by model
  completeness.
\end{proof}

\begin{theorem} \label{thm:unbdd}
  If $K \models C\XF$, then $\Gal(K)$ is unbounded.
\end{theorem}
\begin{proof}
  This follows from Hilbertianity, but here is a direct proof using
  the strategy of Theorem~\ref{Hilb}.
  \begin{claim}
    For any $n$, there are $a_1,\ldots,a_n \in K^\times$ such that
    $\sqrt{a_i/a_j} \notin K$ for $i \ne j$.
  \end{claim}
  \begin{claimproof}
    It suffices to find such $a_i$ in an elementary extension of $K$.
    Build the rational function field $K(t_1,\ldots,t_n)$.  As in the
    proof of Theorem~\ref{Hilb}, there is a regular extension $K^*$ of
    $K(t_1,\ldots,t_n)$ such that $K^* \succeq K$.  Then $t_i/t_j$ is
    not a square in $K^*$ for $i \ne j$, because it is not a square in
    $K(t_1,\ldots,t_n)$.
  \end{claimproof}
  Then there are infinitely many Galois extensions of degree 2.
\end{proof}
\begin{remark} \label{no-imaginaries}
  It follows that $C\XF$ does not eliminate imaginaries.  Indeed, any
  geometric theory of fields which eliminates imaginaries must have
  bounded Galois group.  See~\cite[Corollary 4.2]{geom-field}.
\end{remark}
\begin{theorem} \label{proper-finite-ext}
  If $L/K$ is a proper finite extension of $K$, then $L$ is an
  $\omega$-free PAC field.
\end{theorem}
\begin{proof}
  By Theorem~\ref{PAC}, $L$ is PAC.  Finite separable extensions of
  Hilbertian fields are Hilbertian \cite[Corollary
    12.2.3]{field-arithmetic}, so $L$ is Hilbertian.  Hilbertian PAC
  fields are $\omega$-free~\cite[Theorem A]{Fri-Vol-Hil-PAC}.
\end{proof}
We will see shortly that $\Gal(K)$ is also $\omega$-free, up to
isomorphism.

Recall that a theory has $\mathrm{TP}_2$, the tree property of the
second kind, if there is a formula $\phi(x,y)$ and a model $M$
containing witnesses $a_\eta$ for $\eta : \omega \to \omega$ and
$b_{i,j}$ for $i,j \in \omega$ such that
\begin{equation*}
  M \models \phi(a_\eta,b_{i,j}) \iff \eta(i) = j.
\end{equation*}
\begin{corollary} \label{tp2}
$C\XF$ has $\mathrm{TP}_2$.
\end{corollary}
\begin{proof}
  Suppose $M \models C\XF$.  Take any proper finite extension $L/M$.
  Then $L$ is an $\omega$-free PAC field.  An observation by
  Chatzidakis states that for PAC fields, being unbounded implies
  $\mathrm{TP}_2$ ~\cite[Section 3.5]{zoe-note}.  In particular, $L$
  has $\mathrm{TP}_2$.  As $L$ is interpretable in $M$, $M$ also has
  $\mathrm{TP}_2$.
\end{proof}
Recall that a field $K$ is \emph{very slim} if field-theoretic and model-theoretic algebraic closure agree in $K$ and its elementary extensions~\cite{JK-slim}.
\begin{example}[An algebraically bounded field that is not very slim]\label{eg:alg-bdd}
  Consider the family of genus 2 algebraic curves given
  by \[C_{a,b,c}: y^2=x(x-1)(x-a)(x-b)(x-c),\] for $a,b,c$ distinct
  and different from $0$ and $1$.  Two such curves $C_{a,b,c}$ and
  $C_{a',b',c'}$ over an algebraically closed field $K$ are
  birationally equivalent if and only if some element of $\mathrm{PGL}_2(K)$
  maps $\{0,1,\infty,a,b,c\}$ setwise to $\{0,1,\infty,a',b',c'\}$
  (see the discussion in \cite[Section 2]{igusa-genus-2}, for
  example).  It follows that for fixed $(a_0,b_0,c_0)$, there are no
  more than 720 triples $(a,b,c) \in K^3$ such that $C_{a,b,c}$ and
  $C_{a_0,b_0,c_0}$ are birationally equivalent.
  
  Let $K_0 = \Qq(a_0,b_0,c_0)$, where $a_0,b_0,c_0$ are transcendental
  and independent over $\Qq$, and let $C_0 = C_{a_0,b_0,c_0}$.  Then
  $C_0(K_0)$ is finite by the Mordell conjecture for function fields
  and number fields.  Let $C$ be $C_0$ with the $K_0$-points removed,
  and let $K \supseteq K_0$ be a model of $C\XF$.  We claim that $K$
  is not very slim as a pure field. It suffices to show that the transcendental
  triple $(a_0,b_0,c_0)$ is model-theoretically algebraic over
  $\varnothing$.

  Let $D$ be the set of triples $(a,b,c) \in K^3$ such that the curve
  $C_{a,b,c}$ has only finitely many $K$-points.  This set is
  0-definable in $K$, by elimination of $\exists^\infty$ (see
  Corollary~\ref{exists-infty}).  Then $(a_0,b_0,c_0) \in D$ because
  $C_{a_0,b_0,c_0}(K) = C_0(K)$ equals the finite set $C_0(K_0)$.  It
  remains to show that $D$ is finite.  If $(a,b,c) \in D$, then there
  is a $K$-definable dominant rational map from $C_{a,b,c}$ to $C_0 =
  C_{a_0,b_0,c_0}$, or else axioms $(\mathrm{II}_C)$ would make $C_{a,b,c}(K)$
  be infinite.  The rational map $C_{a,b,c} \dashrightarrow
  C_{a_0,b_0,c_0}$ must be birational, by Riemann-Hurwitz.  By the
  discussion above, there are at most 720 possibilities for $(a,b,c)$,
  and so $|D| \le 720$.
\end{example}

\section{The Galois group}
Continue to fix the field $K_0$, the curve $C$, and the theories $C\XF$ and $C\XF_\forall$, as in the previous section.
\begin{lemma} \label{monster-universal}
  Let $\Mm$ be a monster model of $C\XF$ and let $K$ be a small
  elementary submodel.  Let $F/K$ be a field extension such that $F$
  avoids $C$.  Then there is an embedding $\iota : F \to \Mm$ over $K$
  such that $\iota(F)$ is relatively algebraically closed in $\Mm$.
\end{lemma}
\begin{proof}
  By Theorem~\ref{extender} there is a regular extension $M/F$ such that $M \models
  C\XF$.  Then $M \succeq K$ by model completeness.  Moving everything
  by an isomorphism, we may assume $K \preceq M \preceq \Mm$ by
  saturation.  Then the tower of regular extensions $\Mm \supseteq M
  \supseteq F$ shows that $F$ is relatively algebraically closed in
  $\Mm$.
\end{proof}
For technical reasons, we also need the following confusing fact:
\begin{lemma} \label{mu2}
  In the setting of Lemma~\ref{monster-universal}, suppose we are
  given algebraic closures $K^{\alg}$, $F^{\alg}$, and $\Mm^{\alg}$,
  and choices of embeddings $K^{\alg} \to \Mm^{\alg}$ and $K^{\alg}
  \to F^{\alg}$.  Then there is an embedding $\iota^{\alg} : F^{\alg}
  \to \Mm^{\alg}$ making the diagram commute:
  \begin{equation*}
    \xymatrix{
      & F^{\alg} \ar@{-->}[dr]^{\iota^{\alg}} & \\
      K^{\alg} \ar[ur] \ar[rr] & & \Mm^{\alg} \\
      & F \ar[uu] \ar[dr]^{\iota} & \\
      K \ar[uu] \ar[ur] \ar[rr] & & \Mm \ar[uu]}
  \end{equation*}
\end{lemma}
\begin{proof}
  Choose any embedding $\nu : F^{\alg} \to \Mm^{\alg}$ extending
  $\iota : F \to \Mm$.  The composition \[K^{\alg} \to F^{\alg}
  \stackrel{\nu}{\to} \Mm^{\alg} \tag{$\ast$} \] must have the
  form \[K^{\alg} \stackrel{\sigma}{\to} K^{\alg} \to \Mm^{\alg}\] for
  some $\sigma \in \Gal(K)$.  As $K$ is relatively algebraically
  closed in $\Mm$, the map $\Gal(\Mm) \to \Gal(K)$ is surjective.
  Therefore there is $\tilde{\sigma} \in \Gal(\Mm)$ extending
  $\sigma$, and the map in ($\ast$) can also be written as \[ K^{\alg}
  \to \Mm^{\alg} \stackrel{\tilde{\sigma}}{\to} \Mm^{\alg}.\] Take
  $\iota^{\alg} = \tilde{\sigma}^{-1} \circ \nu$.
\end{proof}
\begin{lemma} \label{k-f-lemma}
  Let $K$ be a model of $C\XF$.  Suppose $\alpha : H \to G$ is a
  surjective homomorphism of finite groups, and $\beta : \Gal(K) \to
  G$ is a continuous, surjective homomorphism.  Then there is an
  extension $F/K$ such that $F$ avoids $C$, and there is a continuous,
  surjective homomorphism $\gamma : \Gal(F) \to H$ making the diagram
  commute:
  \begin{equation*}
    \xymatrix{ \Gal(F) \ar[d] \ar@{-->}[r]^-{\gamma} & H \ar[d]^{\alpha}
      \\ \Gal(K) \ar[r]_-{\beta} & G.}
  \end{equation*}
\end{lemma}
\begin{proof}
  Let $L$ be the fixed field of $\ker(\beta)$.  By Galois theory,
  $L/K$ is a finite Galois extension with $\Gal(L/K) \cong G$.  In
  particular, $G$ acts on $L$.  Let $L(\bar{t}) = L(t_h : h \in H)$ be
  obtained by adjoining $|H|$-many new transcendental elements $t_h$
  to $L$.  The group $H$ acts on $L$ via $\beta : H \to G$, and it
  acts on the generators $t_h$ via $h_2 \cdot t_{h_1} := t_{h_2h_1}$.
  The action of $H$ on $L(\bar{t})$ is faithful.  Let $F$ be the fixed
  field of $H$ in $L(\bar{t})$.  Then $L(\bar{t})/F$ is a finite Galois
  extension with Galois group $H$, and the following diagram commutes:
  \begin{equation*}
    \xymatrix{ \Gal(F) \ar[r] \ar[d] & \Gal(L(\bar{t})/F) \ar[d]
      \ar[r]^-{\cong} & H \ar[d]^{\alpha} \\
      \Gal(K) \ar[r] & \Gal(L/K) \ar[r]^-{\cong} & G.}
  \end{equation*}
  It remains to see that $F$ avoids $C$.  Note that $L(\bar{t})$ is
  the function field of $\Aa^n$ over $L$, where $n = |H|$.  The curve $C$ has positive genus, so there are no non-constant rational maps from $\Aa^1$ to $C$ by Riemann-Hurwitz.  By Lemma~\ref{lem:factor},
  any $L$-definable rational map $\Aa^n \dashrightarrow C$ is
  constant.  In terms of fields (see Fact~\ref{rational-ff}), this means that $C(L(\bar{t})) =
  C(L)$.  Taking $H$-invariants,
  \begin{equation*}
    C(F) = C(L(\bar{t})^H) = C(L(\bar{t}))^H = C(L)^H = C(L^H) = C(K)
    = \varnothing. \qedhere
  \end{equation*}
\end{proof}
\begin{theorem} \label{omega-free}
  If $K \models C\XF$, then $\Gal(K)$ is $\omega$-free: for any
  surjective homomorphism of finite groups $\alpha : H \to G$ and any
  continuous surjective homomorphism $\beta : \Gal(K) \to G$, there is
  a continuous surjective homomorphism $\gamma : \Gal(K) \to H$ making
  the diagram commute:
  \begin{equation*}
    \xymatrix{ \Gal(K) \ar@{-->}[r]^-{\gamma} \ar[dr]_-{\beta} & H \ar[d]^{\alpha} \\ & G.}
  \end{equation*}
\end{theorem}
\begin{proof}
  The property ``$\Gal(K)$ is $\omega$-free'' is a conjunction of first-order sentences, via the usual means of interpreting Galois extensions of $K$ in $K$.  Consequently, we
  may replace $K$ with a monster model $\Mm$.
  
  We are given a
  continuous surjective homomorphism $\beta : \Gal(\Mm) \to G$.  Let
  $\mathbb{L}$ be the fixed field of $\ker(\beta)$.  By Galois theory,
  $\mathbb{L}/\Mm$ is a finite Galois extension, $\Gal(\mathbb{L}/\Mm)
  \cong G$, and $\Gal(\Mm) \to G$ is the composition
  \begin{equation*}
    \Gal(\Mm) \to \Gal(\mathbb{L}/\Mm) \stackrel{\cong}{\to} G.
  \end{equation*}
  In particular, there is a natural action of $G$ on $\mathbb{L}$.

  Let $K \preceq \Mm$ be a small model over which the extension
  $\mathbb{L}$ is defined.  Then there is a finite Galois extension
  $L/K$ such that $\mathbb{L} = L \otimes_K \Mm$ and $\Gal(\mathbb{L}/\Mm)
  \cong \Gal(L/K) \cong G$.  Then $G$ again acts on $L$.  Thus $\beta
  : \Gal(\Mm) \to G$ factors through $\Gal(K)$:
  \begin{equation*}
    \xymatrix{ \Gal(\Mm) \ar[dr]^{\beta} \ar[d] & \\ \Gal(K) \ar[r] & G.}
  \end{equation*}
  By Lemma~\ref{k-f-lemma}, there is an extension $F/K$ and a
  commuting diagram
  \begin{equation*}
    \xymatrix{ \Gal(F) \ar[d] \ar[r] & H \ar[d]^{\alpha}
      \\ \Gal(K) \ar[r] & G.}
  \end{equation*}
  where the horizontal arrows are surjective.  Moreover,
  Lemma~\ref{k-f-lemma} lets us arrange for $C(F) = \varnothing$.  By
  Lemma~\ref{monster-universal}, we may move $F$ by an isomorphism,
  and assume that $K \subseteq F \subseteq \Mm$ with $F$ relatively
  algebraically closed in $\Mm$.  Then we get a commuting diagram
  \begin{equation*}
    \xymatrix{ \Gal(\Mm) \ar[dr]^{\gamma} \ar[d] & \\ \Gal(F) \ar[r] \ar[d] & H
      \ar[d]^{\alpha} \\ \Gal(K) \ar[r] & G.}
  \end{equation*}
  All the maps are surjective, because $\Mm/F$ and $\Mm/K$ are regular
  extensions.  This completes the proof.
\end{proof}
\begin{remark}
  Technically we also need to apply Lemma~\ref{mu2} in the proof of
  Theorem~\ref{omega-free}.  The problem is that $\Gal(K)$ depends
  functorially on the pair $(K,K^{\alg})$, and cannot be made to
  depend functorially on $K$ alone.  Before moving $F$ we already have
  maps $\Gal(\Mm) \to \Gal(K)$ and $\Gal(F) \to \Gal(K)$, which must have
  come from embeddings $(K,K^{\alg}) \to (\Mm,\Mm^\alg)$ and
  $(K,K^{\alg}) \to (F,F^{\alg})$.  The proof of
  Theorem~\ref{omega-free} implicitly assumes that the following diagram
  commutes:
  \begin{equation*}
    \xymatrix{ & \Gal(F) \ar[dr] & \\ \Gal(K) \ar[rr] \ar[ur] & & \Gal(\Mm).} \tag{$\ast$}
  \end{equation*}
  In order for this to work, we need a commuting diagram of pairs
  \begin{equation*}
    \xymatrix{ & (F,F^{\alg}) \ar[dr] & \\ (K,K^\alg) \ar[rr] \ar[ur]
      & & (\Mm,\Mm^{\alg}),}
  \end{equation*}
  which is what Lemma~\ref{mu2} provides.  To put it another way, if
  we chose the embedding $F^{\alg} \to \Mm^{\alg}$ carelessly, then
  the diagram ($\ast$) might be off by an inner automorphism.
\end{remark}

\section{NSOP$_4$} \label{nsop4-section}
Recall that for $n \ge 3$, a theory has $\mathrm{SOP}_n$ if there is a formula $\phi(x,y)$ and a model $M$ containing witnesses $a_0, a_1, a_2, a_3, \ldots$, such that
\begin{equation*}
  M \models \phi(a_i,a_j) \text{ for } i < j,
\end{equation*}
but the partial type
\begin{equation*}
  \{\phi(x_1,x_2),\phi(x_2,x_3),\ldots,\phi(x_{n-1},x_n),\phi(x_n,x_1)\}
\end{equation*}
is inconsistent.  Otherwise, the theory is $\NSOP_n$.  In this section, we show that $C\XF$ has $\NSOP_4$.  For simplicitly, we restrict to the case where $K_0 = \Qq$, though the proof easily generalizes.

Let $\omega\PAC_0$ be the theory of $\omega$-free PAC fields of
characteristic 0.  Recall the $Sol_n$ predicates from
Corollary~\ref{cor:near-qe}.  Let $\mathcal{L}_{Sol}$ be the language
of rings expanded by $Sol_n$. We need the following facts; (1) is~\cite[Theorem~27.2.3]{field-arithmetic} and (2) follows from~\cite{CH-bounded}.
\begin{fact} \phantomsection \label{pac-facts-A}
  \begin{enumerate}
  \item $\omega\PAC_0$ is the model companion of the theory of fields
    of characteristic 0 in the language $\mathcal{L}_{Sol}$.
  \item Model-theoretic algebraic closure agrees with field-theoretic
    relative algebraic closure in $\omega\PAC_0$.
  \end{enumerate}
\end{fact}
Note that the statement of (2) also holds with $C$XF in place of $\omega\PAC_0$ by Theorem~\ref{acl}.  From
Fact~\ref{pac-facts-A}(1), we see that for any field $K$ of
characteristic 0, there is some $M \models \omega\PAC_0$ and an
embedding $K \to M$ of $\mathcal{L}_{Sol}$-structures.  Equivalently,
every field of characteristic 0 has a \emph{regular} extension that is
an $\omega$-free PAC field.  Our strategy for proving NSOP$_4$ for
$C\XF$ is to embed a model $M \models C\XF$ in a regular extension
$\hat{M} \models \omega\PAC_0$ and use Chatzidakis's work on
independence in $\omega\PAC_0$.
\begin{fact} \label{pac-facts-B}
  $\omega\PAC_0$ has quantifier elimination in $\mathcal{L}_{Sol}$.
\end{fact}
This appears in unpublished notes of Cherlin, van den Dries and
Macintyre~\cite[Theorem~41]{reg-closed}, but it can also be deduced easily from
Fact~\ref{pac-facts-A}(1) via the proof strategy of
Corollary~\ref{cor:near-qe}, as follows.
\begin{proof}[Proof of Fact~\ref{pac-facts-B}]
  Suppose $M_1, M_2$ are models of $\omega\PAC_0$, and $A_i$ is a
  substructure of $M_i$ for $i = \{1,2\}$.  Suppose $f : A_1 \to A_2$
  is an isomorphism in the language $\mathcal{L}_{Sol}$.  We must show
  that $f$ is a partial elementary map.  As in the proof of
  Corollary~\ref{cor:near-qe}, we reduce to the case where $A_i$ is a
  subfield, relatively algebraically closed in $M_i$.  Moving by an
  isomorphism, we may assume $A_1 = A_2 =: A$ and $f = \id_A$.  Then
  $M_1$ and $M_2$ are regular extensions of $A$, so $\Frac(M_1
  \otimes_A M_2)$ is a regular extension of $A$, $M_1$, and $M_2$.
  Take a regular extension $N$ of $\Frac(M_1 \otimes_A M_2)$ such that
  $N \models \omega\PAC_0$, as discussed above.  Then $N$ is a regular
  extension of $M_i$ for $i = 1,2$, so the embedding $M_i \to N$ is an
  embedding of $\mathcal{L}_{Sol}$-structures.  Then $M_1 \preceq N
  \succeq M_2$ by model completeness of $\omega\PAC_0$
  (Fact~\ref{pac-facts-A}(1)), and we conclude that $\id_A$ is a
  partial elementary map from $M_1$ to $M_2$.
\end{proof}
  
Next, we turn to the notion of independence in $\omega\PAC_0$.  For a
given sufficiently saturated model $M \models \omega\PAC_0$ and small
subsets $A,B,C\subseteq M$, we use $\ind^{\ACF}$ to denote forking
independence in $\ACF$, namely $A\ind^{\ACF}_C B$ holds if $A\ind_C B$
holds in $M^{\alg}$. The following definition comes
from~\cite[(1.2)]{zoe-omega-forking}, where it is called \emph{weak
independence}.

For (model-theoretically) relatively algebraically closed subsets $A,B\supseteq C$ in $M$, we define $A\ind^{I}_C B$ to mean that 
\[
A\ind^{\ACF}_C B \text{ and } A^{\alg}B^{\alg}\cap M =AB.
\]
By \cite[Lemma~2.8(3)]{zoe-omega-forking} this is equivalent to
\[
A\ind^{\ACF}_C B \text{ and } B^{\alg} \cap A^{\alg}\acl(AB) = C^{\alg}B.\]
For general subsets $A,B,C\subseteq M$, $A
\ind^I_C B$ is defined to be 
\[
\acl(AC)\ind^I_{\acl(C)}\acl(BC).
\]
\begin{fact}
When $M$ is a sufficiently saturated $\omega$-free PAC field of characteristic $0$, $\ind^I$
satisfies the following properties. 
\begin{itemize}
\item Invariance: $A\ind^I_C B$ iff $
\sigma(A)\ind^I_{\sigma(C)}\sigma(B)$ for $\sigma \in \Aut(M)$.
\item Symmetry: $A\ind^I_C B$ iff $B\ind^I_C A$.
\item Existence: $A \ind^I_B B$ always holds.
\item Monotonicity: $A\ind^I_C B$ and $A_0\subseteq A, B_0\subseteq B$ implies that $A_0\ind^I_C B_0$.
\item Strong finite character: If $a
  \nind^I_c b$ for some tuples $a, b, c$ (possibly infinite) then
  there is a formula $\phi(x,b,c) \in \tp(a/bc)$ such that $a'
  \nind^I_c b$ for any $a'$ satisfying $\phi(x,b,c)$.
\end{itemize}
\end{fact}
Strong finite character is \cite[Lemma~6.2]{artem-nick}, and the other properties are easily verified from the two definitions of $\ind^I$ given above.
\begin{lemma} \label{indiscernible-independence}
  If $a_1, a_2, a_3, \ldots, b_1, b_2$ is an indiscernible sequence in
  an $\omega$-free PAC field $M$, then $b_1 \ind^I_{\bar{a}} b_2$.
\end{lemma}
\begin{proof}
  The type $\tp(b_1/b_2\bar{a})$ is finitely satisfiable in $\bar{a}$,
  so this follows as in the proof of
  \cite[Proposition~5.8]{artem-nick}.  In more detail, supose for the
  sake of contradiction that $b_1 \nind^I_{\bar{a}} b_2$.  By strong
  finite character there is $\phi(x,b_2,\bar{a})$ in
  $\tp(b_1/b_2\bar{a})$ such that any $c$ realizing
  $\phi(x,b_2,\bar{a})$ satisfies $c \nind^I_{\bar{a}} b_2$.  But by
  indiscernibility, $M \models \phi(a_i,b_2,\bar{a})$ for sufficiently
  large $i$, and so $a_i \nind^I_{\bar{a}} b_2$.  But $b_2
  \ind^I_{\bar{a}} \bar{a}$ by existence, $b_2 \ind^I_{\bar{a}} a_i$
  by monotonicity, and $a_i \ind^I_{\bar{a}} b_2$ by symmetry, a
  contradiction.
\end{proof}
We will also need the following independence theorem for $\ind^I$, which is equivalent to part of \cite[Lemma~3.2]{zoe-omega-forking}
\begin{fact} \label{ind-theorem}
Let $M$ be a sufficiently saturated $\omega$-free PAC field of characteristic 0.  Suppose $A$ is a small relatively algebraically closed subfield of $M$, and $\alpha, \alpha'', \beta, \beta', \gamma', \gamma''$ are tuples such that
\begin{gather*}
    \alpha \equiv_A \alpha'', \qquad \beta \equiv_A \beta', \qquad \gamma' \equiv_A \gamma'' \\
    \alpha \ind^I_A \beta, \qquad \beta' \ind^I_A \gamma', \qquad \alpha'' \ind^I_A \gamma''.
\end{gather*}
Then there are tuples $\alpha^\dag, \beta^\dag, \gamma^\dag$ such that $\alpha^\dag \ind^I_A \beta^\dag \gamma^\dag$ and
\begin{equation*}
\alpha^\dag \beta^\dag \equiv_A \alpha \beta, \qquad \beta^\dag \gamma^\dag \equiv_A \beta' \gamma', \qquad \alpha^\dag \gamma^\dag \equiv_A \alpha'' \gamma'' .
\end{equation*}
\end{fact}
Aside from $\ind^I$, we will need the stronger relation $\ind^{III}$
from \cite[(3.6)]{zoe-note}.  Let $M$ be a model of $\omega\PAC_0$.  If $C \subseteq A, B$ are relatively algebraically closed subsets of $F$, define $A \ind^{III}_C B$ to mean
\begin{equation*}
  A \ind^{\ACF}_C B \text{ and } M \cap (AB)^{\alg} = AB.
\end{equation*}
For general subsets $A, B, C$, define $A \ind^{III}_C B$ to mean
\begin{equation*}
  \acl(AC) \ind^{III}_{\acl(C)} \acl(BC).
\end{equation*}
\begin{fact} \label{iii-exist}
  Suppose $M$ is sufficiently saturated.  For any small sets $A, B,
  C$, there is $A' \equiv_C A$ such that $A' \ind^{III}_C B$.
  
  Moreover, if $C = \acl(C)$, then $\tp(A'/BC)$ is uniquely
  determined: if $A'' \equiv_C A$ and $A'' \ind^{III}_C B$, then $A''
  \equiv_{BC} A'$.
\end{fact}
Fact~\ref{iii-exist} is mentioned in \cite[(3.6)]{zoe-note} and is straightforward to prove.

\begin{theorem} \label{nsop4}
  $C\XF$ has NSOP$_4$.
\end{theorem}
\begin{proof}
  Throughout the
  proof, work in $\mathcal{L}_{Sol}$, the language with the $Sol_n$
  solvability predicates.  In particular, $\qftp$ will mean
  quantifier-free type in $\mathcal{L}_{Sol}$.  As noted above, both
  $C\XF$ and $\omega\PAC_0$ have quantifier elimination in this
  language.

  Suppose for the sake of contradiction that there is an instance of
  SOP$_4$ in a highly saturated model $M \models C\XF$.  Specifically,
  suppose there is a sequence $a_1, a_2, a_3, \ldots$ and a formula
  $\phi(x,y)$ such that $M \models \phi(a_i,a_j)$ for $i < j$, and the
  partial type
  \begin{equation*}
    \{\phi(x_1,x_2), \phi(x_2,x_3), \phi(x_3,x_4), \phi(x_4,x_1)\}
  \end{equation*}
  is inconsistent.  Without loss of generality, $a_1, a_2, \ldots$ is
  indiscernbile.  Extend to a longer indiscernible sequence $a_1, a_2,
  \ldots, b_1, b_2, \ldots$.
  \begin{claim} \label{endgame}
    It suffices to find $L \models C\XF_\forall$ containing $c_1,c_2,c_3,c_4$
    such that
    \begin{equation*}
      \qftp^L(c_ic_{i+1}) = \qftp^M(b_1b_2) \text{ for } i \in \Zz/4\Zz.
    \end{equation*}
  \end{claim}
  \begin{claimproof}
    By Theorem~\ref{extender}, there is a regular extension $L'/L$
    with $L \models C\XF$.  The embedding $L \mapsto L'$ respects the
    $Sol_n$ predicates, so it preserves quantifier-free types.
    Replacing $L$ with $L'$, we may therefore assume $L \models C\XF$.
    Then $\tp^L(c_ic_{i+1}) = \tp^M(b_1b_2) = \tp^M(a_1a_2)$ by
    quantifier elimination in $C\XF$.  It follows that $L \models
    \phi(c_i,c_{i+1})$, so the partial type $\{\phi(x_1,x_2),
    \phi(x_2,x_3), \phi(x_3,x_4), \phi(x_4,x_1)\}$ is realized in $L$.  The complete type of a tuple contains the complete theory of the underlying structure, so $\tp^L(c_1c_2) = \tp^M(b_1b_2)$ implies $L \equiv M$.  Then
    the partial type $\{\phi(x_1,x_2),
    \phi(x_2,x_3), \phi(x_3,x_4), \phi(x_4,x_1)\}$ is also realized in $M$, a contradiction.
  \end{claimproof}
  Embed $M$ into a highly saturated regular extension $\hat{M} \models
  \omega\PAC_0$.  We work in $\hat{M}$ until stated otherwise.  The sequence
  $a_1, a_2, \ldots, b_1, b_2, \ldots$ remains indiscernible in $\hat{M}$
  because of quantifier elimination in $\omega\PAC_0$ in $\mathcal{L}_{Sol}$.  Let $A_0 =
  a_1a_2a_3\ldots$ and $A = \acl(A_0)$.  By
  Lemma~\ref{indiscernible-independence}, $b_1 \ind^I_{A_0} b_2$,
  or equivalently $b_1 \ind^I_A b_2$.  The sequence $b_1, b_2, \ldots
  $ is $A_0$-indiscernible, which implies it is $A$-indiscernible.

  Let $p = \tp(b_1/A) = \tp(b_2/A)$.  By Fact~\ref{iii-exist}, there
  exist $e_1, e_2$ such that $\tp(e_1/A) = \tp(e_2/A) = p$ and $e_1
  \ind^{III}_A e_2$.  By the symmetry and stationarity of $\ind^{III}$, it follows that $e_1e_2 \equiv_A e_2e_1$.

  By the independence theorem (Fact~\ref{ind-theorem}, applied to $\alpha = b_1$, $\beta = b_2$, $\beta' = b_1$, $\gamma' = b_2$, $\alpha'' = e_1$ and $\gamma'' = e_2$), we can find
  $c_1, c_2, c_3$ such that
  \begin{gather*}
    c_1c_2 \equiv_A b_1b_2 \text{ and so } c_1 \ind^I_A c_2 \\
    c_2c_3 \equiv_A b_1b_2 \text{ and so } c_2 \ind^I_A c_3 \\
    c_1c_3 \equiv_A e_1e_2 \text{ and so } c_1 \ind^{III}_A c_3 \\
    c_1 \ind^I_A c_2c_3.
  \end{gather*}
  Note that the second and fourth lines ensure that $c_1$, $c_2$, and $c_3$ are
  ACF-independent over $A$.
  \begin{claim} \label{step-1}
    The field $\acl(Ac_1c_2)\acl(Ac_2c_3)$ satifies $C\XF_\forall$ (it
    avoids $C$).
  \end{claim}
  \begin{claimproof}
    From the ACF-independence of $c_1, c_2, c_3$ over $A$, it follows that
    \begin{equation*}
      \acl(Ac_1c_2) \ind^{\ACF}_{\acl(Ac_2)} \acl(Ac_2c_3)
    \end{equation*}
    By Lemma~\ref{lem:regular-tensor}, it suffices to show that
    $\acl(Ac_1c_2) \models C\XF_\forall$ and $\acl(Ac_2c_3) \models C\XF_\forall$.  As $c_1c_2
    \equiv_A c_2c_3 \equiv_A b_1b_2$, it suffices to show that
    $\acl(Ab_1b_2) \models C\XF_\forall$. But
    \begin{equation*}
      \acl(Ab_1b_2) = \acl_{\hat{M}}(Ab_1b_2) = \acl_{\hat{M}}(A_0b_1b_2) =
      \acl_M(A_0b_1b_2) \subseteq M,
    \end{equation*}
    because $\hat{M} \models \omega\PAC_0$ was a regular extension of $M
    \models C\XF$.  Then $M \models C\XF_\forall \implies \acl(Ab_1b_2) \models C\XF_\forall$.
  \end{claimproof}
  Now $c_1c_3 \equiv_A e_1e_2$ and $e_1e_2 \equiv_A e_2e_1$.  It
  follows that $c_1c_3 \equiv_A c_3c_1$.  Take $\sigma \in \Aut(\hat{M}/A)$
  exchanging $c_1$ and $c_3$.  Letting $c_4 = \sigma(c_2)$, we get
  \begin{gather*}
    c_3c_4 = \sigma(c_1c_2) \equiv_A c_1c_2 \equiv_A b_1b_2 \\
    c_4c_1 = \sigma(c_2c_3) \equiv_A c_2c_3 \equiv_A b_1b_2 \\
    \acl(Ac_3c_4)\acl(Ac_4c_1) \models C\XF_\forall.
  \end{gather*}
  All of these properties are statements about $\tp(c_4/Ac_1c_3)$, so
  moving $c_4$ by an automorphism over $Ac_1c_3$, we may further
  assume that
  \begin{gather*}
    c_4 \ind^{III}_{Ac_1c_3} c_2 \tag{$\ast$}
  \end{gather*}
  by Fact~\ref{iii-exist}.
  \begin{claim}
    \label{final-claim}
    The field $L := \acl(Ac_1c_2)\acl(Ac_2c_3)\acl(Ac_3c_4)\acl(Ac_4c_1)$
    satisfies $C\XF_\forall$ (it avoids the curve $C$).
  \end{claim}
  \begin{claimproof}
    Note that ($\ast$) implies $c_4 \ind^{\ACF}_{Ac_1c_3} c_2$ and then
    \begin{equation*}
      \acl(Ac_3c_4)\acl(Ac_4c_1) \ind^{\ACF}_{\acl(Ac_1)\acl(Ac_3)}
      \acl(Ac_1c_2)\acl(Ac_2c_3).
    \end{equation*}
    The fact that $c_1 \ind^{III}_A c_3$ implies by definition that $\acl(Ac_1)\acl(Ac_3)$ is relatively algebraically closed in $\hat{M}$, and so
    \begin{equation*}
        \acl(Ac_3c_4)\acl(Ac_4c_1) \text{ is a regular extension of } \acl(Ac_1)\acl(Ac_3).
    \end{equation*}
    The same holds for $\acl(Ac_1c_2)\acl(Ac_3c_3)$.
    Then Lemma~\ref{lem:regular-tensor} applies, and it suffices to show that the
    two fields
    \begin{gather*}
      \acl(Ac_3c_4)\acl(Ac_4c_1) \\
      \acl(Ac_1c_2)\acl(Ac_2c_3)
    \end{gather*}
    avoid $C$, which was arranged above via Claim~\ref{step-1} and
    the choice of $c_4$.
  \end{claimproof}
  With $L$ as in Claim~\ref{final-claim},
  \begin{equation*}
    L \cap \langle Ac_ic_{i+1} \rangle^{\alg} = \acl_{\hat{M}}(Ac_ic_{i+1})
  \end{equation*}
  for $i \in \Zz/4\Zz$, because $L$ contains
  $\acl_{\hat{M}}(Ac_ic_{i+1})$.  Therefore,
  \begin{equation*}
    \qftp^L(Ac_ic_{i+1}) = \qftp^{\hat{M}}(Ac_ic_{i+1}) =
    \qftp^{\hat{M}}(Ab_1b_2) = \qftp^M(Ab_1b_2).
  \end{equation*}
  By Claim~\ref{endgame}, we are done.
\end{proof}

\section{An instance of SOP$_3$}
Recall the definition of SOP$_3$ from Section~\ref{nsop4-section}.
\begin{theorem} \label{thm:some-sop3}
  Suppose $K_0 = \Qq$ and $C$ is the curve $x^4 + y^4 + 1 = 0$.  Let
  $K$ be a model of $C\XF$ such that $\Abs(K) = \Qq^{\alg} \cap \Rr$.  Then $K$ has
  SOP$_3$.
\end{theorem}
\begin{proof}
  Let $\phi(x,y)$ be the formula saying that $y-x$ is a non-zero
  fourth power.  We claim that the formula $\phi$ witnesses SOP$_3$ in
  $K$.  First of all, there is a sequence $\{a_i\}_{i < \omega}$ such
  that $K \models \phi(a_i,a_j)$ for all $i < j$, namely $a_i = i$.
  Indeed, if $i < j$, then $\sqrt[4]{j-i} \in \Qq^{\alg} \cap \Rr
  \subseteq K$, so $K \models \phi(i,j)$.  On the other hand,
  \begin{equation*}
    \{\phi(x_1,x_2),\phi(x_2,x_3),\phi(x_3,x_1)\}
  \end{equation*}
  is inconsistent, or else
  \begin{gather*}
    x_2 - x_1 = t_1^4 \\
    x_3 - x_2 = t_2^4 \\
    x_1 - x_3 = t_3^4
  \end{gather*}
  for some non-zero $t_1,t_2,t_3$.  Then
  \begin{equation*}
    t_1^4 + t_2^4 + t_3^4 = (x_2 - x_1) + (x_3 - x_2) + (x_1 - x_3) = 0,
  \end{equation*}
  and $(t_1/t_3,t_2/t_3)$ is a point on $C$.
\end{proof}
Recall that a structure or theory is \emph{strictly NSOP$_4$} if it is
NSOP$_4$ but not NSOP$_3$.  For example, the Henson graph is strictly
NSOP$_4$.
\begin{corollary} \label{strict-nsop4}
  There is a field $(K,+,\cdot)$ that is strictly NSOP$_4$ in the
  language of fields.
\end{corollary}
We have learned through private communication that Ramsey recently has found constructions of strictly NSOP$_4$ PAC fields via coding graphs in the absolute Galois group. We believe that no other examples of strictly NSOP$_4$ fields are known.
\begin{conjecture}
  Any model of $C\XF$, for any $C$ and $K_0$, is strictly NSOP$_4$.
  Equivalently, any model of $C\XF$ has SOP$_3$.
\end{conjecture}

\section{Potential generalizations of $C\XF$}
The construction of $C\XF$ should have a number of generalizations and
variants.
\begin{enumerate}
\item If $V$ is a variety without rational points, we can look at the
  class of fields $K$ such that $V(K) = \varnothing$.  What conditions
  on $V$ ensure that the model companion exists?  The construction in
  Sections~\ref{sec:axiom}--\ref{sec:inter} mainly relies on the fact that rational maps
  into $V$ have bounded complexity, and this probably holds for some
  higher-dimensional varieties.  For example, can we let $V$ be a
  surface of general type?
\item Let $C_1, \ldots, C_n$ be a finite list of smooth irreducible
  curves, each with genus $\ge 2$.  The class of fields $K$ such that
  $C_1(K) = C_2(K) = \cdots = C_n(K) = \varnothing$ probably has a
  model companion.  This is an instance of the previous point, with $V
  = C_1 \cup C_2 \cup \cdots \cup C_n$.
\item Let $C_1, C_2, \ldots$ be an \emph{infinite} list of smooth
  irreducible curves, each with genus $g_{C_i} \ge 2$.  Consider the
  class of fields $K$ such that every $C_i(K)$ is empty.  If $\lim_{i
    \to \infty} g_{C_i} = \infty$, then a model companion should
  exist.  The key observation is that the set
  \begin{equation*}
    \{i < \omega : \text{there is a dominant rational map } V
    \dashrightarrow C_i\}
  \end{equation*}
  is finite and bounded by the complexity of $V$.  Axiom (I$_C$) seems
  problematic\footnote{One would need to say something like, ``If
  $L/K$ is a finite extension, then $C_i(L) \ne \varnothing$ for
  \emph{some} $i < \omega$.'' This is not a first-order property.}, but
  can be omitted for reasons related to Corollary~\ref{drop-axiom}(2).  When combined with the method of Example~\ref{eg:alg-bdd}, this should give examples of algebraically bounded fields in which $\dcl(\varnothing)$ has infinite transcendence degree.
\item In a different direction, if $C$ is a curve over $\Qq$ of genus
  $\ge 2$ without rational points, then the theory of \emph{ordered}
  fields avoiding $C$ should have a model companion.  The
  axiomatization would be similar to the axiomatization of $C\XF$, but
  changing axiom (II$_C$) to only apply to varieties $V$ containing a
  smooth point in the real closure of $K$.\footnote{If $V$ is an
  irreducible variety over an ordered field $K$, then $V$ has a smooth
  point in the real closure if and only if there is an order on $K(V)$
  extending the order on $K$.}  Similarly, there may be variants of
  this construction involving valued fields or formally $p$-adic
  fields.
\item In another direction, one could attempt to generalize the
  construction of $C\XF$ from algebraically closed fields to other
  strongly minimal or $\omega$-stable theories.  For example, let $T$
  be an $\omega$-stable theory with quantifier elimination, and let
  $\phi(x)$ be a strongly minimal formula.  Consider the class of
  structures
  \begin{equation*}
    \{K : M \models T, ~ K \subseteq M, ~ K = \dcl(K), ~ \phi(K) =
    \varnothing\}.
  \end{equation*}
  This class is elementary and inductive.  When does it have a model
  companion?  The arguments for $C$XF seem to generalize,
  \emph{provided} that the following conditions hold:
  \begin{itemize}
  \item $T$ has elimination of imaginaries.  This lets us reason about
    ``finite extensions'' in a coherent way.
  \item $T$ has the definable multiplicity property.  More precisely,
    we need the class of definable sets of Morley degree 1 to be
    ind-definable (a small union of definable families).  This allows
    one to talk about ``irreducible varieties''.
  \item Given a definable set $V$ of Morley degree 1, we need an
    effective bound on the complexity of generic maps $V
    \dashrightarrow C$.
  \end{itemize}
  We don't know how to find this combination of properties in other
  $\omega$-stable theories.  It seems worth looking into DCF and its
  reducts, as well as Hrushovski fusions of ACF with other theories.
\end{enumerate}
We leave these questions to future papers.
\begin{remark}
  The class of PAC fields is much richer than the class of
  algebraically closed fields, and one might hope to complete the
  analogy
\begin{quote}
  ACF : PAC :: $C$XF : ``pseudo $C$XF''.
\end{quote}
Presumably, ``pseudo $C$XF'' fields would have more variability in their Galois
groups.

The obvious strategy would be to look at fields that avoid $C$, and
are e.c.\@ in \emph{regular} extensions $L/K$ which avoid $C$.
However, Corollary~\ref{drop-axiom}(2) shows that this is equivalent
to $C$XF, rather than a new theory.  In particular, the Galois groups
would still be $\omega$-free groups, by Theorem~\ref{omega-free}.  Consequently, there doesn't seem to be a meaningful generalization in this direction.
\end{remark}

\begin{acknowledgment}
  The first author was supported by the National Natural Science
  Foundation of China (Grant No.\@ 12101131) and the Ministry of Education of China (Grant No.\@ 22JJD110002). Part of the work was completed when the second author was a postdoctoral fellow funded by Fondation Sciences Math\'ematiques de Paris.  

  Both of the authors would like to express their gratitude to Erik Walsberg, from whom the project benefited significantly during its early stages.  We would also like to thank the two anonymous referees, who caught numerous errors. The second author would like to thank Ariyan Javanpeykar for helpful comments and discussions.
\end{acknowledgment}

\bibliographystyle{amsalpha}
\bibliography{ref}

\end{document}